\documentclass[letterpaper,11pt]{amsart}

\usepackage[utf8]{inputenc}
\usepackage[english]{babel}
\usepackage[margin = 1.25in]{geometry} 
\usepackage{amsthm}
\usepackage{amsmath}
\usepackage{amssymb}
\usepackage[scr]{rsfso}
\usepackage{csquotes}
\usepackage{bm,bbm}
\usepackage{color}
\usepackage[dvipsnames]{xcolor}
\usepackage{graphicx}
\usepackage{caption, subcaption}
\usepackage{array}
\usepackage{tikz}
\usetikzlibrary{matrix,arrows}
\usepackage{tikz-cd}
\usepackage[framemethod=TikZ]{mdframed}
\usepackage{youngtab}
\usepackage{mathtools}
\usepackage[normalem]{ulem}
\usepackage[all]{xy}
\usepackage{rotating}
\usepackage{extarrows}
\usepackage{accents}
\usepackage{leftidx}
\usepackage{upgreek}
\usepackage{xr-hyper}
\usepackage[draft]{todonotes}
\usepackage{proofread}
\usepackage{hyperref}
\usepackage{macros}
\usepackage{qtree}

\hypersetup{  
  pdftitle    = {Boundedness of semistable sheaves},
  pdfauthor   = {Dylan Spence},
  colorlinks=true,
  citecolor=orange,
  linkcolor=blue}

\numberwithin{equation}{section}
\title[Boundedness of semistable sheaves]{Boundedness of semistable sheaves}
\begin{document}
\begin{abstract}
    In this expository article, we follow Langer's work in \cite{langer} to prove the boundedness of the moduli space of semistable torsion-free sheaves over a projective variety, in any characteristic. 
\end{abstract}




\author{Haoyang Guo}
  \address{Max Planck Institute for Mathematics, Vivatsgasse 7, Bonn, 53115, Germany}
  \email{hguo@mpim-bonn.mpg.de}
  
\author{Sanal Shivaprasad}
  \address{Department of Mathematics, University of  Michigan, 530~Church Street, Ann Arbor, MI~48109, United States}
  \email{sanal@umich.edu}
  
\author{Dylan Spence}
  \address{Department of Mathematics, Indiana University, Rawles Hall, 831 East 3rd Street, Bloomington, IN~47405, United States}
  \email{dkspence@indiana.edu}

\author{Yueqiao Wu}
  \address{Department of Mathematics, University of  Michigan, 530~Church Street, Ann Arbor, MI~48109, United States}
  \email{yueqiaow@umich.edu}
\maketitle
\setcounter{secnumdepth}{1}
\setcounter{tocdepth}{1}
\tableofcontents

\section{Introduction}

The study of moduli is one of the oldest branches of algebraic geometry and forms one of the central pillars of our modern understanding. Historically, such questions date back to Riemann, who, in his pursuit to understand what we now call Riemann surfaces, determined that a complex projective curve of genus $g$ depends on exactly $3g-3$ parameters (which he referred to as "moduli" - hence the name). The central question of moduli theory has not evolved much; we are interested in whether or not various algebraic objects can be "parametrized" (in some sense) by some other algebraic object. For example, one might be interested in forming a moduli space of algebraic curves with some fixed genus, or in our case, a moduli space of torsion-free sheaves with fixed Hilbert polynomial on a fixed variety $X$.  

To construct the moduli space of torsion-free sheaves $\{E_\alpha\}$ with fixed Hilbert polynomial $P$ over a given variety $X$, one of the first and the most fundamental problems is the \emph{boundedness} of the collection $\{E_\alpha\}$.
This property is equivalent to the moduli space, if exists, being a finite type scheme over the base field, and thus a reasonable geometric object that one can work with. However it was quickly discovered that even with fixing the Hilbert polynomial, certain pathological examples prevent such a moduli space from being well-behaved. Mumford, in the case of curves, introduced the notion of \emph{semistability} for a torsion-free sheaf as a solution for this problem, and to demonstrate its efficacy, he showed that the collection of semistable vector bundles of a fixed rank and degree on a fixed curve is bounded. As the notion of semistability can be generalized to torsion-free sheaves over general projective varieties, it is natural to ask about the boundedness for semistable sheaves on higher dimensional varieties. Following Langer's work in \cite{langer}, the goal of our article is to give a positive answer to the boundedness problem.
Precisely, we prove the following:

\begin{theorem}\label{main}
Let $\mathcal{M}^P(X)$ be the moduli space of semistable torsion-free sheaves with fixed Hilbert polynomial $P$. Then it is of finite-type.
\end{theorem}

Let us now discuss the idea of the proof. Given any (semistable) torsion-free sheaf $E$ in our collection, we write $E|_H$ for the restriction of $E$ to a general hypersurface $H$, and $\mu_{\max}(E|_H)$ for the maximal slope in the Harder-Narasimhan polygon of $E|_H$. 

Our most important tool in the proof is Kleiman's criterion (Theorem \ref{Kleiman}) and an accompanying inequality (Lemma \ref{H0BoundedByMuMax}), which, when combined, say that to prove the boundedness of the family of torsion-free sheaves $\mathcal{M}^P(X)$, it suffices to give a uniform bound of $\mu_{\max}(E|_H)$. Because of this fact, most of the work in proving boundedness is producing such a uniform bound, which we obtain in Corollary \ref{RestrictionLEstimate} and Theorem \ref{RestrictionSlopeEstimate}.
Explicitly, it can be understood as follows:
\begin{theorem}\label{inequality}
Let $E$ be a torsion-free sheaf over $X$, with Hilbert polynomial $P$, and denote the maximal slope of its Harder-Narasimhan polygon be $\mu_0$.
Then we have
\[
\mu_{\max}(E|_H) \leq C_1(P) + C_2(P)\mu_0,
\]
where $H$ is a general hypersurface of $X$, and $C_i(P)$ are constants determined only by $P$ and $X$.
\end{theorem}
The key is that this result allows us to bound the slopes $\mu_{\max}(E|_H)$ of the restrictions $E|_H$ by the slopes of the original $E$, plus extra terms and coefficients that are controlled only by the Hilbert polynomial of $E$.

The major contents of our article are then devoted to obtaining the inequality above, following \cite[Section 3]{langer}.
The strategy can be understood as a double induction on two different results.
The first one, Theorem \nameref{Res}, is a special case of Theorem \ref{inequality} above, considering the invariants before and after the restriction of $E$ to a hypersurface, where $E$ is torsion-free of rank $\leq r$ (c.f. \cite[Theorem 3.1]{langer}).
It relates the slopes of the Harder-Narasimhan filtration of the restriction $E|_H$ to the discriminant and Harder-Narasimhan filtration of $E$ itself.
The second result is a collection of several various numerical inequalities, which we refer to as \emph{Bogomolov's inequalities} (Theorem \nameref{BI}, c.f. \cite[Theorem 3.2-3.4, and $T^5(r)$]{langer}), due to their similarity with the well-known inequality in characteristic zero of the same name.

Up to an extra term, it says that the discriminant $\Delta(E)$ of $E$ is non-negative when the torsion-free sheaf $E$ is (strongly) semistable, where $E$ is of rank $\leq r$.
Using the notations above, Langer's induction schema can be summarized as the following two implications (c.f. Section \ref{sec Res}):
\label{schema}
\begin{equation}
    \begin{cases}
        BI(r) \Rightarrow Res(r);&\\
        Res(r) + BI(r) \Rightarrow BI(r+1).&
    \end{cases}
\end{equation}

We note that both results are true automatically for rank $r=1$. In this way, the two technical results are proved together, thus so is the inequality in Theorem \ref{inequality} and the boundedness in Theorem \ref{main}.

We also mention that in the positive characteristic case, a Frobenius pullback of a semistable sheaf may no longer be semistable. So we define semistable sheaves whose Frobenius pullbacks are all semistable as being \emph{strongly semistable}. A priori, for a given sheaf, the Harder-Narasimhan filtrations of its Frobenius pullbacks could be unrelated to each other. One of Langer's key observations is that their differences are controllable: the Harder-Narasimhan filtrations eventually stabilize with respect to the Frobenius (Theorem \ref{fdHN}, \cite[Theorem 2.7]{langer}). This allows us to pass between semistable and strongly semistable sheaves in the proof.

Lastly, we briefly mention the history preceding Langer's result; a more detailed introduction can be found in \cite{langer}.
Theorem \ref{main} was known in characteristic zero, proven by Barth, Spindler, Maruyama, Forster, Hirschowitz and Schneider. In positive characteristic however, only the cases for curves and surfaces were known.
At that time, it was known that understanding the numerical quantities associated to $E$  before and after restricting $E$ to a hypersurface (often called Grauet--M\"ulich type results) should allow one to prove boundedness for a given family, but prior to \cite{langer} only a coarser result by Mehta and Ramanathan was available.
On the other hand, Bogomolov showed the non-negativity of the discriminant of $E$ whenever $E$ is semistable, assuming $X$ is in characteristic zero. It was unknown as well whether the result is completely true in positive characteristic, but it was clear that some version would be needed. Thus Langer's key contributions were in developing positive characteristic versions of the above, and also in combining them in a fruitful way.

\subsection{Leitfaden of the article.}
We start with Section \ref{Sec pre} on necessary preliminaries.
The section contains three subsections, including basics on the stability of coherent sheaves and additional results in positive characteristic.
Moreover, as the proof of main theorems requires working with polarizations consisting of nef divisors, we also include a subsection on how to approximate the nef polarizations by ample polarizations.
In Section \ref{Sec BI} we introduce several forms of Bogomolov's inequalities in positive characteristic, and show that they are equivalent in Theorem \nameref{BI}.
For the reader's convenience, this section corresponds to the implications $T^5(r) \Rightarrow T^3(r) \Rightarrow T^4(r) \Rightarrow T^2(r) \Rightarrow T^5(r)$ in \cite[$\S$3.6 - $\S$3.8]{langer}.
In Section \ref{sec Res}, we prove Theorem \ref{Res}, and complete the major technical part of the article by proving the induction schema as in (\ref{schema}).
This corresponds to \cite[$\S$3.5, $\S$3.9]{langer} and \cite{langer-erratum}.
At last, in Section \ref{Sec bound}, we combine the main technical results above to show the boundedness of the moduli space $\mathcal{M}^P(X)$, thus finishing the proof of Theorem \ref{main}.

\subsection{Acknowledgements}
We would like to express our gratitude to the organizers of the Stacks Project Workshop for organizing the event online during the pandemic. We thank Alex Perry for kindly guiding us through the paper during the week and beyond. We also thank Faidon Andriopoulos for helpful discussions during the workshop, and thank Adrian Langer for answering technical questions. At last, we thank the referee for reading the draft carefully and proposing various comments to help improve the article. 

Haoyang Guo was partially funded by the FRG grant no. DMS-1952399 during the writing of the project.

\section{Preliminaries}\label{Sec pre}
In this section, we provide the preliminaries for the article.

\allowdisplaybreaks
\subsection{Stability of Coherent Sheaves}
In this section we recall some useful facts on the (semi)stability  of coherent sheaves. We refer to \cite{huybrechts} for details and proofs. Let $X$ be a smooth projective variety of dimension $n$ over an algebraically closed field. Fix $n-1$ ample divisors $D_1,\ldots,D_{n-1}$ on $X$. Let $E$ be a coherent sheaf on $X$. Recall that a sheaf $E$ is said to be pure (of dimension $d\leq \dim X$) if for all nontrivial subsheaves $f \subset E$, $\dim (F) := \dim \operatorname{supp}(F) = d$.

\begin{definition}
The \emph{slope} of $E$, with respect to the polarization $(D_1,\ldots,D_{n-1})$, is defined by
$$\mu(E):= \frac{\deg(E)}{\mathrm{rk}(E)}:=\frac{c_1(E)\cdot D_1\cdots D_{n-1}}{\mathrm{rk}(E)}$$
where $\mathrm{rk}(E)$ is the rank of $E$ at the generic point. If $\mathrm{rk}(E) = 0$, then the $\mu(E) := \infty$.
\end{definition}

\noindent As indicated, the slope of a coherent sheaf does depend on the choice of ample (or nef) divisors defining a polarization. For the purpose of clarity, we will not reference this choice of polarization when discussing slope unless there is a serious risk of confusion.

\begin{definition}
We say that $E$ is \emph{(semi)stable} if $E$ is pure and for any nonzero proper subsheaf $F\subset E$, $\mu(F)< \mu(E)$ (resp. $\mu(F)\leq \mu(E)$). 
\end{definition}

\noindent This notion of (semi)stability is more generally known as slope-(semi)stability or $\mu$-(semi)stability. An equivalent formulation of the definition can also be stated in terms of quotient sheaves.
\begin{lemma} \label{EquivChar}\cite[Proposition 1.2.6]{huybrechts}
Let $E$ be a coherent sheaf.
Then $E$ is semistable if and only if for all the quotient sheaves $E\to G$, $\mu(E)\le \mu(G).$
\end{lemma}

\noindent This formulation of the definition gives a quick proof of the following very useful fact.
\begin{lemma}\label{HomVan}\cite[Proposition 1.2.7]{huybrechts}
Let $E_1$ and $E_2$ be semistable sheaves with slopes $\mu_1$ and $\mu_2$ such that $\mu_1 > \mu_2$. Then, $\mathrm{Hom}(E_1,E_2) = 0$. 
\end{lemma}

The definition of semistability also behaves well with regards to exact sequences, as the following two lemmas indicate.
\begin{lemma}\label{ExactAdd}
Let $E' \to E \to E''$ be a short exact sequence of torsion-free coherent sheaves, and let $\mu',\mu,\mu''$ and $r',r,r''$ be the slopes and the ranks of $E', E, E''$ respectively. Then, $r\mu = r'\mu' + r''\mu''$. In particular if the slopes of two of $E',E,E''$ are the same, then so is the case for the third. 
\end{lemma}

\begin{lemma}
If $E' \to E \to E''$ is a short exact sequence and let $\mu',\mu,\mu''$ denote the slopes of $E',E,E''$ respectively, then 
\begin{itemize}
    \item If $E$ is semistable and either $\mu' = \mu$ or $\mu'' = \mu$, then $E'$ and $E''$ are also semistable. 
    \item If $E'$ and $E''$ are semistable and $\mu' = \mu''$, then $E$ is semistable with $\mu = \mu' = \mu''$. 
\end{itemize}
\end{lemma}

One of the more important technical tools is the Harder-Narasimhan filtration, which is defined below.
\begin{definition}
A \emph{Harder-Narasimhan filtration} for $E$ is an increasing filtration
\[0= E_0 \subset E_1\subset \cdots \subset E_d =E\]
such that the factors $F_i:=E_i/E_{i-1}, i = 1, \cdots, d$ are semistable sheaves with slopes $\mu_i$ satisfying 
\[\mu_{\max}(E):=\mu_1 > \mu_2 >\cdots >\mu_d =: \mu_{\min}(E).\]
\end{definition}

\begin{proposition}\cite[Theorem 1.3.4]{huybrechts}\label{HNF existunique}
Every torsion-free sheaf $E$ has a unique Harder-Narasimhan filtration.
\end{proposition}
\begin{remark}
Just like the slope, the Harder-Narasimhan filtration is also dependent on the choice of the polarization $(D_1,\ldots,D_{n-1})$, and could be different with respect to polarizations. 
\end{remark}
\begin{remark}
We should also remark here that in the construction of the Harder-Narasimhan filtration for a torsion-free sheaf, an important step is establishing the existence and uniqueness of a \emph{maximal destabilizing subsheaf}. We use this notion a few times, so we give its definition here. Given a torsion-free sheaf $E$, then the \emph{maximal destabilizing subsheaf} $F \subset E$ is a semistable coherent subsheaf such that $\mu(F) \geq \mu(G)$ for all other subsheaves $G \subset E$, and moreover if $\mu(F) = \mu(G)$, then $F \supset G$. 
\end{remark}

Next, we wish to introduce the key tools of the paper, Kleiman's criterion and a related inequality. For the conclusion of Kleiman's criterion to make sense however, we should remind the reader of the technical property of boundedness.

\begin{definition}
Let $M$ be a set of coherent sheaves on $X$. Then $M$ is said to be \emph{bounded} if there is a scheme $B$ of finite-type and a coherent sheaf $F$ on $X \times B$ with $M \subset \{F_b \,| \, b \in B \text{ closed } \}$. Here $F_b$ is the pullback of $F$ along $X \times \{b\} \to X \times B$.
\end{definition}

\noindent Kleiman's criterion provides a very convenient way of determining whether or not a given family is bounded. 

\begin{theorem}[Kleiman's criterion] \cite[Theorem 1.7.8]{huybrechts} \label{Kleiman}
Let $\{E_\alpha\}$ be a family of coherent sheaves over $X$ with the same Hilbert polynomial $P$.
Then the family is bounded if and only if there are constants $C_i$, for $i=0,\ldots, \deg(P)$, such that for every $E_\alpha$ there exists an $E_\alpha$-regular sequence of hyperplanes $H_1,\ldots, H_{\deg(P)}$, satisfying 
\[
h^0(E_\alpha|_{\cap_{j\leq i}H_j}) \leq C_i, \forall i.
\]
\end{theorem}

Here we recall that a hyperplane $s \in H^0(X,\mathcal{O}_X(1))$ is said to be \emph{$E$-regular} if the map $$E(-1) \overset{\cdot s}{\to} E$$ is injective. A sequence $\{s_1,...,s_l\} \subset H^0(X,\mathcal{O}_X(1))$ is $E$-regular if $s_i$ is $E/(s_1,..,s_{i-1})E(-i)$-regular for all $1\leq i\leq l$.

\begin{lemma}
\label{H0BoundedByMuMax}\cite[Lemma 3.3.2]{huybrechts}
Let $E$ be a torsion-free sheaf of rank $r$. Then for any $E$-regular sequence of hyperplane sections $H_1, \cdots, H_n$, the following inequality holds for $i=1, \cdots, n$:
\[ \frac{h^0(X_i, E_i)}{r\deg(X)} \leq \frac 1{i!}\left[\frac{\mu_{\max} (E_1)}{\deg(X)}+i\right]_+^i,\]
where $X_i\in |H_1|\cap \cdots \cap |H_{n-i}|$, $E_i = E|_{X_i}$, and $[x]_+ = \max\{ 0,x \}$ for any real number $x$. 
\end{lemma}

Combining the previous two results, provided we can get a uniform bound on $\mu_{\max}$, will prove the boundedness for semistable sheaves with a fixed Hilbert polynomial. To get such estimates, we will often use the following lemma. 
\begin{lemma}
\label{SESMuMaxBound}
If $0 \to A  \to B \to C \to 0$ is an exact sequence of torsion-free sheaves, then $ \mu_{\max}(B) \leq \max \{ \mu_{\max}(A), \mu_{\max}(C) \}$. 
\end{lemma}

At the end of the subsection, we introduce the following definition that helps visualize the Harder-Narasimhan filtration.
\begin{definition}[Harder-Narasimhan polygon]
\label{DefinitonHNP}
Consider a torsion-free sheaf $E$ on $X$ with its Harder-Narasimhan filtration $0 = E_0 \subset \dots \subset E_m = E.$ Let $p(E_i) = (\mathrm{rk}(E_{i}),\deg(E_i))$. We define the \emph{Harder-Narasimhan polygon} of $E$, denoted as $\mathrm{HNP}(E)$, to be the convex hull of the points $p(E_0),\dots,p(E_m)$ in $\mathbb{R}^2$. See Figure \ref{fig:HNPExample}.
\end{definition}

\begin{figure}[h!]
\begin{tikzpicture}
\node at (0,0) [below] {$p(E_0) = (0,0)$};
\fill (0,0) circle (2pt);

\draw (0,0) -- node [above=2pt,sloped] {\footnotesize slope = $\mu(E_1/E_0)$} (3,3);

\node at (3,3) [left] {$p(E_1)$};
\fill (3,3) circle (2pt);

\draw (3,3) -- node [above=2pt, sloped] {\footnotesize slope = $\mu(E_2/E_1)$} (6,4);

\node at (6,4) [above] {$p(E_2)$};

\fill (6,4) circle (2pt);
\draw [dashed] (6,4) -- (8,2);
\node at (8,2) [right] {$p(E_m)$};
\fill (8,2) circle (2pt);
\draw (8,2) -- node [below=2pt,sloped] {\footnotesize slope = $\mu(E)$} (0,0);
\end{tikzpicture}
\caption{An illustration of the Harder Narasimhan Polygon}
\label{fig:HNPExample}
\end{figure}
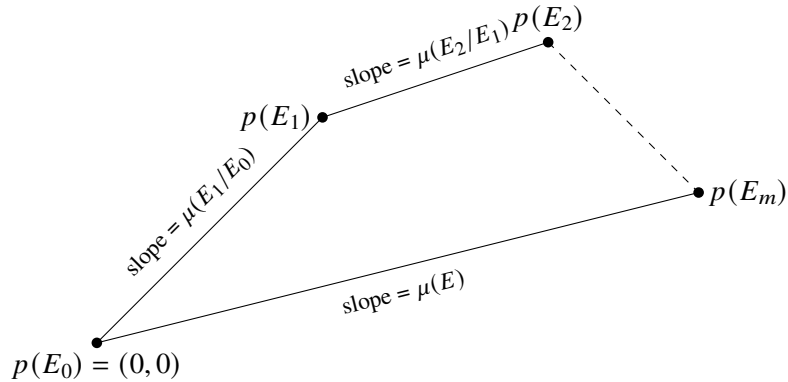

In fact, the Harder-Narasimhan polygon is a convex polygon with vertices $p(E_0), \dots, p(E_m)$, and the slope of the line segment $\overline{p(E_i) p(E_{i+1})}$ is $\mu(E_i/E_{i+1})$.

\begin{remark}
\label{RemarkHNPLiesAboveSubsheaves}
If $F \subset E$ is a torsion-free subsheaf $F$ inside of a semistable sheaf $E$, then the point $p(F) = (\mathrm{rk}(F),\deg(F))$ lies below the polygon $\HNP(E)$ i.e.~for any $x\in [0,\mathrm{rk}(F)]$, we have
\[
\sup\{y~|~(x,y)\in \HNP(F)\} \leq \sup\{y~|~(x,y)\in \HNP(E)\}.
\]

\end{remark}

\subsection{Approximation of nef polarizations}\label{subsec approx}
In this subsection, we approximate the Harder-Narasimhan filtration for a polarization consisting of \emph{nef divisors}, that will be used later.
As this is a nonstandard setup and does not appear in most of the literature, we provide the full details of the proof.
We assume the existence and uniqueness of the (absolute) Harder-Narasimhan filtration with respect to an \emph{ample polarization} as in last subsection.

Our first main result this subsection is the following approximation result.
\begin{theorem}\label{approx}
	Let $(D_1,\ldots,D_{n-1})$ be a set of nef divisors, and $H$ an ample divisor.
	Assume $E$ is a torsion-free coherent sheaf over $X$.
	Then there exists a positive number $\epsilon$, and a (unique) filtration of saturated subsheaves 
	\[
	0=E_0\subset E_1\subset \cdots E_l=E,
	\]
	such that for any $t\in (0,\epsilon)$, the above is the Harder-Narasimhan filtration for the ample polarization $(D_1+tH,\ldots, D_{n-1}+tH)$.
\end{theorem}

\begin{proof}
	Denote $H_i(t)$ to be the $\mathbb{R}$-divisor $D_i+tH$, which is ample \cite[Corollary 1.4.10]{lazarsfeld}.
	For a torsion-free sheaf $F$, we denote $\deg_t(F)$ and $\mu_t(F)$ to be the degree and slope of $F$ with respect to the polarization $(H_1(t),\ldots, H_{n-1}(t))$.
	Note that when $t=0$, we get exactly the degree and slope with respect to the polarization $(D_1,\cdots ,D_{n-1})$, which we write in short as $\deg(F)$ and $\mu(F)$ separately.
	
	We first notice that it suffices to show the following: there exists $\epsilon>0$ and a saturated subsheaf $E_1\subset E$, such that $E_1$ is the maximal destabilizing subsheaf of $E$ for $(H_1(t),\ldots, H_{n-1}(t))$ and any $t\in (0,\epsilon)$.
	The uniqueness of $E_1$ follows from the uniqueness of the Harder-Narasimhan filtration for the ample polarization $(H_1(t),\cdots H_{n-1}(t))$ as in Proposition \ref{HNF existunique}.
	
	To find such $E_1$, let $\mathcal{C}$ denote the set of torsion-free subsheaves of $E$ and consider the following map of sets
	\begin{align*}
		\mu_t(-):\mathcal{C} & \longrightarrow \{\text{Degree } n-1 \text{ polynomials in }t\};\\
		E' &\longmapsto \mu_t(E')=\frac{c_1(E')H_1(t)\cdots H_{n-1}(t)}{\mathrm{rk}(E')}.
	\end{align*}
	For each $E'\subset E$, the image $\mu_t(E')$ is the polynomial $a_0+a_1t+\cdots+ a_{n-1} t^{n-1}$, where we have
	\begin{align*}
	&a_0(E')=\mu(E');\\
	&a_1(E')=\sum_{1\leq i\leq n-1} \mu_{(H,D_1,\ldots, \hat{D}_i,\ldots,D_{n-1})} (E').\\
	\end{align*}
	For general $j\leq n$, the coefficient $a_j$ is the sum of slopes of $E'$ with respect to all possible choice of polarizations $(H,\ldots,H,D_{i_1},\ldots,D_{i_{n-1-j}})$, such that each of the first $j$ entries of the polarization are all equal to $H$.
	Moreover, each polynomial $\mu_t(E')$ has coefficients in $\frac{1}{r!}\mathbb{Z}$, whose coefficients are finite linear combinations of slopes of $E'$ for a fixed, finite number of choices of polarization.
	In particular, for each $0\leq i\leq n$, the collection of coefficients $\{a_i(E')~|~E'\subset E\}$ is bounded above.
	\footnote{To see this, we first note that since $a_i(E')$ is a finite positive linear combination of $\mu_{H^i, D_{j_1},\ldots,D_{j_{n-1-i}}}(E')$, it suffices to bound each  slope for the given nef polarization. The latter can  be proved by induction on ranks of $E$ as in the classical case: To check the case when $E$ is a line bundle, it suffices to show the set of slopes of sub line bundles is bounded by slope of $E$ itself, which is true again by approximation via adding each $D_{j_l}$ by $tH$ and making $t$ approach to zero. Here we observe that the inequality holds for any $t>0$, so by the continuity, we get the upper bound.
	The general case of $E$ follows from the induction hypothesis on $E'$ and $E''$ in a short exact sequence of vector bundles $E'\rightarrow E\rightarrow E''$. }
	Here we denote the subset of polynomials consisting of the image of the map $\mu_t(-)$ by $\mathcal{D}$. 
	
	We then define a lexicographic order on elements in $\mathcal{D}$ as follows:
	\begin{multline*}
	W_1(t)=\sum^d a_i t^i < W_2(t)=\sum^d b_i t^i \text{ if } \\ a_0= b_0,\cdots, a_{i-1}=b_{i-1}, a_i < b_i \text{ for some } i.    
	\end{multline*}
	Since the coefficients of polynomials in $\mathcal{D}$ are bounded above, we can find the maximum polynomial $P(t)$ in $\mathcal{D}$. Moreover, since the degree of a sheaf is a discrete quantity, we can find a subsheaf $E_1$ in the preimage of $\mu_t^{-1}(P(t))$ whose rank is maximal.
	
	Finally, we prove that $E_1$ is exactly the maximal rank destabilizing subsheaf of $E$ with respect to $(H_1(t),\ldots,H_{n-1}(t))$ for $t$ small enough.
	To show this, it suffices to show the following, whose complete details are left to the reader:
	\begin{claim}
		Let $\mathcal{D}$ be a set of degree $n$ polynomials with coefficients bounded above, and whose coefficients are in $\frac{1}{N}\mathbb{Z}$ for some positive integer $N$. Then a polynomial $P(t)\in \mathcal{D}$ is maximum with respect to the lexicographic order if and only if for small enough $t>0$ $P(t)>Q(t)$ (in $\mathbb{R}$) for all other $Q(t) \in \mathcal{D}$.
	\end{claim}
The idea of the claim is the following: as $t$ approaches zero, higher power terms are dominated by those with lower powers, so the polynomial that has the largest first several terms under the lexicographic order will also have maximum value in $\mathbb{R}$ for $t$ small enough.
\end{proof}
The above filtration of subsheaves $\{E_i\}$ in general fails to be the Harder-Narasimhan filtration with respect to $(D_1,\ldots, D_{n-1})$ for $t=0$.
However, the filtration is in fact a \emph{weak Harder-Narasimhan filtration} with respect to $(D_1,\ldots, D_{n-1})$,
in the sense that each $E_i/E_{i-1}$ is semistable with respect to $(D_1,\ldots, D_{n-1})$, and we have inequalities
\[
\mu(E_{i+1}/E_i) \geq \mu(E_i/E_{i-1}), \,\, 1\leq i\leq l-1.
\]
Note that we do not have strict inequalities here, which is part of the definition of the Harder-Narasimhan filtration.

The semistability of the subsheaves $E_i$ can be seen as follows:
\begin{lemma}\label{semi approx}
	Let $E'$ be a subsheaf of the torsion-free sheaf $E$ such that for $t>0$ small enough, $E'$ is semistable with respect to $(H_1(t),\ldots, H_{n-1}(t))$.
	Then $E'$ is semistable with respect to $(D_1,\ldots, D_{n-1})$.
\end{lemma}
\begin{proof}
	Assume $E'$ admits a subsheaf $E'_1$ whose slope with respect to $(D_1,\ldots, D_{n-1})$ is strictly larger than that of $E'$.
	By the continuity of the polynomial $\mu_t$ in $t$, we get the inequality $\mu_t(E'_1)> \mu_t(E')$ for $t>0$ small enough, which contradicts our assumption.
\end{proof}
The same strategy above in fact implies the existence and uniqueness of the Harder-Narasimhan filtration for nef divisors, generalizing the result for ample ones.
\begin{corollary}
Let $(D_1,\ldots,D_{n-1})$ be a set of nef divisors of $X$.
Then any torsion-free coherent sheaf $E$ over $X$ admits a unique Harder-Narasimhan filtration with respect to the polarization $(D_1,\ldots,D_{n-1})$.
\end{corollary}
\begin{proof}
	Let $H$ be a fixed ample divisor over $X$.
	By Theorem \ref{approx}, there exists a unique filtration of saturated subsheaves $E_i$ of $E$, such that when $t>0$ is very small, $E_i$ is the Harder-Narasimhan filtration of $E$ with respect to $(H_1(t),\ldots, H_{n-1}(t))$.
	Each $E_i$ is semistable with respect to $(D_1,\ldots,D_{n-1})$ by Lemma \ref{semi approx}, and the slopes of $E_i/E_{i-1}$ is non-strictly decreasing.
	Thus by picking the subsheaves in this filtration so that the slopes of the factors become strictly decreasing, we get the Harder-Narasimhan filtration of $E$ with respect to $(D_1,\ldots,D_{n-1})$.
	At last, the uniqueness of the Harder-Narasimhan filtration follows from the uniqueness of the maximal destabilizing subsheaf with respect to $(D_1,\ldots, D_{n-1})$, where the latter can be checked via the continuity of $\mu_t$.
	So we are done.
\end{proof}

At last, we give a simple example illustrating the degeneration of the Harder-Narasimhan filtration.
\begin{example}
	Let $X=\mathbb{P}^1\times \mathbb{P}^1$ be the product of two projective lines over a field $k$.
	Let $D$ be the trivial divisor, $H$ be the ample divisor of bidegree $(1,2)$ over $X$, and let $H(t):=D+tH$ be the sum.
	Denote $L_1=(0,1)$ and $L_2=(1,0)$ to be the two line bundles over $X$, and let $E$ be the the direct sum $L_1\oplus L_2$.
	Then for any $t>0$, the divisor $H(t)$ is ample in $X$ as it can be written as a sum of the pullback of ample divisors on two separate factors of $X$.
	Moreover, for each $t>0$, we have $L_1\cdot H(t)=2t >t=L_2\cdot H(t)$, and thus 
	\[
	0\subset L_1 \subset E \tag{$\ast$}
	\]
	is the Harder-Narasimhan filtration of $E$ for the ample polarization $H(t)$, for any $t>0$.
	However, when $t=0$, as $H(0)=D=0$ is the trivial divisor,
	any subsheaf of $E$ has the same slope $0$.
	In particular, the Harder-Narasimhan filtration of $E$ is the trivial filtration, and the filtration $(\ast)$ becomes a filtration of subsheaves whose graded pieces have the same slopes.
	
\end{example}


\subsection{Positive characteristic} \label{positive char}
\newcommand{\Spec}{\mathrm{Spec}\ }
\newcommand{\Frob}{\mathrm{Frob}}
\newcommand{\can}{\mathrm{can}}

The main difference in characteristic $p$ is that we need to work with the notion of \emph{strong semistability}, as the Frobenius pullbacks of semistable sheaves are not necessarily semistable. 

First, we recall a few basic notions from algebraic geometry in positive characteristic. Let $k$ be an algebraically closed field of characteristic $p > 0$. 
Let $X$ be a smooth projective $k$-variety. 
The \emph{absolute Frobenius} morphism $F_X : X \to X$ is the map on $X$ given locally on an open subset $\Spec R \subset X$ by $a \mapsto a^p$. For simplicity of notation, we just denote $F_X$ by $F$. 
Note that $F$ is not a map of $k$-schemes. 

We also fix nef divisors $D_1,\dots,D_{n-1}$ on $X$ and we will compute slope with respect to the polarization $(D_1,\ldots,D_{n-1})$. 

\begin{definition}[Strong semistability] 
A coherent sheaf  $E$ on $X$ is said to be \emph{strongly semistable} if  $(F^{e})^*E$ is a semistable sheaf on $X$ for all $e \geq 0$. 
\end{definition}
\begin{remark}
A coherent sheaf that is semistable but not strongly semistable can be found for example in \cite[Corollary 2]{Br05}.
\end{remark}

In the positive characteristic, instead of just keeping track of $\mu_{\min}$ and $\mu_{\max}$ of a coherent sheaf, we also keep track of a few other invariants related to the Frobenius pullbacks. We note first that Frobenius pullback alters the slope of a sheaf, indeed we have that $\mu((F^e)^*E) = p^e \mu(E)$. For this reason, we define the related quantity
$$ L_{\max}(E) := \lim_{e \to \infty} \frac{\mu_{\max}({(F^e)^*E})}{p^e}.$$

Note that the sequence $\frac{\mu_{\max}({(F^e)^*E})}{p^e}$ is increasing in $e$, thus the limit $L_{\max}(E)$ exists in $\mathbb{R} \cup \{ \infty \}$. We will show later that this limit is indeed finite.  
Similarly, we can define $L_{\min}$. By definition, we have $L_{\max}(E) \geq \mu_{\max}(E)$ and $L_{\min}(E) \leq \mu_{\min}(E)$. 
It immediately follows from the definition that if $E$ is strongly semistable, then $$L_{\max} = \mu_{\max} = \mu(E) = \mu_{\min}(E) = L_{\min}(E).$$

Let us set 
$$ \alpha(E) := \max \{L_{\max}(E) - \mu_{\max}(E), \mu_{\min}(E) - L_{\min}(E)\} .$$

We would like to find an upper estimate for $\alpha(E)$. To do this, we first state a theorem on how to detect instability of the Frobenius pullback of a semistable sheaf. For details, refer to \cite[Section 2]{langer}. The idea is to make use of the canonical connection $F^*E \to F^*E \otimes \Omega_{X}$ on the Frobenius pullback.

\begin{theorem}
\label{nonTrivMapOnHNPieces}
Let $E$ be a semistable sheaf on $X$ such that $F^*E$ is not a semistable sheaf on $X$. Let $0 = E_0 \subset E_1 \subset \dots \subset E_m = F^*E$ be the Harder-Narasimhan filtration of $F^*E$. Then, the natural $\mathcal{O}_X$ homomorphisms $E_i \to (E/E_i) \otimes \Omega_X$ induced by the canonical connection are non-zero.
\end{theorem}

We now estimate the the minimum and maximum slopes of the Frobenius pullback of a semistable sheaf. The idea is to use the non-zero maps above along with Theorem \ref{nonTrivMapOnHNPieces} to get an estimate of how far apart the slopes can be. 

\begin{lemma}
Let $A$ be a nef divisor such that $T_X(A)$ is globally generated (or equivalently if $\Omega \hookrightarrow \mathcal{O}_X(A)^{\oplus l}$ for some $l$) and let $E$ be a torsion-free semistable sheaf on $X$. Then, 
$$\mu_{\max}(F^*E) - \mu_{\min}(F^*E) \leq (\mathrm{rk}(E)-1)A\cdot D_1 \cdot \dots \cdot D_{n-1}.$$
\end{lemma}
\begin{proof}
Let $0 = E_0 \subset E_1 \subset \dots \subset E_m = F^*E$ be the Harder-Narasimhan filtration of $F^*E$. Using Theorem \ref{nonTrivMapOnHNPieces}, we have that the $\mathcal{O}_X$-homomorphism $E_i \to F^*E/E_i \otimes \Omega_X$ is non-zero. Thus, we get that $$ \mu(E_i/E_{i-1}) =\mu_{\min}(E_i) \leq \mu_{\max}(F^*E/E_i \otimes \Omega_X). $$

Since $\Omega \hookrightarrow \mathcal{O}_X(A)^{\oplus l}$, we get that $\mu_{\max}(F^*E/E_i \otimes \Omega_X) \leq \mu_{\max}(F^*E/E_i \otimes \mathcal{O}_X(A)) = \mu(E_{i+1}/E_{i}) + A\cdot D_1 \cdot \dots \cdot D_{n-1}$.  

Thus, we get that 
$$ \mu(E_i/E_{i-1}) - \mu(E_{i+1}/E_{i}) \leq  A\cdot D_1 \cdot \dots \cdot D_{n-1}$$

Summing this inequality, we get the result. 
\end{proof}

\begin{proposition}
If $A$ is a nef divisor such that $T_X(A)$ is globally generated and $E$ is a torsion-free sheaf on $X$, then 
$$ \frac{\mu_{\max}(F^*E)}{p} \leq \mu_{\max}(E) +  \frac{(\mathrm{rk}(E)-1)}{p}A \cdot D_1\dots D_{n-1}$$
\end{proposition}
\begin{proof}
Let $0 = E_0 \subset E_1 \subset \dots \subset E_m = E$ be the Harder-Narasimhan filtration of $E$. 
Applying the previous proposition to $F^*(E_i/E_{i-1})$, we get that 
$$ \frac{\mu_{\max}(F^*(E_i/E_{i-1}))}{p} \leq \mu(E_i/E_{i-1}) +\frac{(\mathrm{rk}(E)-1)}{p}A \cdot D_1\dots D_{n-1} .$$
Note that $F^*E_i$ form a filtration of $F^*E$ and thus by Lemma \ref{SESMuMaxBound}  $\mu_{\max}(F^*E) \leq \max_i \{ \mu_{\max}(F^*(E_i/E_{i-1})) \}$. Using this, we get the required result.
\end{proof}

We have the following bound on $\alpha(E)$ (which in particular shows that $L_{\max}(E)$ and $L_{\min}(E)$ are finite).

\begin{proposition}\label{bound of alpha E}
If $A$ is a nef divisor such that $T_X(A)$ is globally generated, then 
$$ \alpha(E) \leq \frac{(\mathrm{rk}(E)-1)}{p-1}A \cdot D_1\dots D_{n-1}.$$
\end{proposition} 
\begin{proof}
Applying induction to the previous proposition, we see that 
$$\frac{\mu_{\max}((F_X^e)^*E)}{p^e} \leq \mu_{\max}(E) +  (\mathrm{rk}(E)-1)\left( \frac{1}{p} + \frac{1}{p^2} + \dots + \frac{1}{p^e}\right) A \cdot D_1\dots D_{n-1}.$$
Letting $e \to \infty$, we get the required result. 
\end{proof}

\subsection{Finite determinancy of the Harder-Narasimhan filtration.}
We will now prove the following theorem. 
\begin{theorem}\label{fdHN}
For every torsion-free sheaf $E$, there exists a non-negative integer $e_0$ such that all the factors in the Harder-Narasimhan filtration of $(F^{e_0})^* E$ are strongly semistable. 
\end{theorem}

The proof uses the Harder-Narasimhan polygon (see Definition \ref{DefinitonHNP}). For any sheaf $G$ over a smooth projective variety $X$, let us define $p(G) := (\mathrm{rk} G, \deg G) \in \mathbb{R}^2$. 
 We also define 
$$ \mathrm{HNP}_e(E) := \{ (x,y) \in \mathbb{R}^2 \mid (x,p^e y) \in \mathrm{HNP}((F^{e})^*E) \}.$$

Note that $\mathrm{HNP}_e(E)$ forms an increasing sequence of convex subsets of $\mathbb{R}^2$ and we define $\mathrm{HNP}_\infty(E) := \overline{\bigcup_e \mathrm{HNP}_e(E)}$.

\begin{proposition}
$HNP_{\infty}(E)$ is a bounded convex polygon. 
\end{proposition}
\begin{proof}
It is clear that $\mathrm{HNP}_{\infty}(E)$ is convex as each of the $\mathrm{HNP}_e(E)$ are convex. To see that it is bounded, note that the rank coordinates of $\mathrm{HNP}_e(E)$ lie in in the interval $[0,\mathrm{rk}(E)]$. The fact that $\alpha(E)$ is finite also tells us that there is a uniform bound on the degree coordinates of $\mathrm{HNP}_e(E)$. Thus we see that $\mathrm{HNP}_\infty(E)$ is a bounded subset of $\mathbb{R}^2$. 

To show $\HNP_{\infty}(E)$ is a polygon, we claim that it is a convex hull of the set $\{(r,q_{\infty,r}) \mid  r \in \{0,\dots, \mathrm{rk}(E)\}$, where  $q_{\infty,r} = \sup \{d \mid (r,d) \in \HNP_{\infty}(E) \}$. Note that $q_{\infty,0} = p(0) = (0,0)$ and $q_{\infty,\mathrm{rk}(E)} = p(E) = (\mathrm{rk}(E),\deg(E))$.

It is clear that the convex hull of these points is contained in $\HNP_{\infty}(E)$. To see the converse, first note that all $\HNP_k(E)$ and thus $\HNP_\infty(E)$ lie above the line segment joining $p(E_0) = (0,0)$ and $p(E) = (\mathrm{rk}(E),\deg(E))$. 
It is thus enough to show that every vertex of $\HNP_e(E)$ lies in the convex hull for all $e$. Pick any vertex $(r,d)$ of $\HNP_e(E)$ for some $e$ and some $r \in \{ 0,\dots, \mathrm{rk}(E)\}$. Then, since $(r,d)$ lies above the line segment $\overline{p(0)p(E)}$ and below the point $(r, q_{r})$, one can easily see that $(r,d)$ lies in the convex hull of $p(0),p(E)$ and $(r,q_{r})$. 
\end{proof}

Before proving the finite determinancy of the Harder-Narasimhan filtration, let us introduce some notation. Let $0= E_{0,e} \subset \dots \subset E_{m_e,e}$ denote the Harder-Narasimhan filtration of $(F^e)^*E$. Let $p_{i,e} = (\mathrm{rk}(E_{i,e}), \\ \deg(E_{i,e})/p^e)$ denote the vertices of $\HNP_e(E)$. 

Let $(0,0) = p_{0,\infty},\dots,p_{s,\infty} = p(E)$ denote the vertices of the $\HNP_\infty(E)$ and denote the coordinates of $p_{j,\infty}$ as $p_{j,\infty} = (r_{j,\infty}, d_{j,\infty})$. Let us denote the slope of the line segment $\overline{p_{(j-1), \infty}p_{j, \infty}}$ as $\mu_{j,\infty}$. For every $j$, there exists a sequence $p_{i_j,e} \to p_{j,\infty}$ for some sequence $i_j$ as $e \to \infty$. 

\begin{proof}[Proof of Theorem \ref{fdHN}]
We prove that there exists some $e_0$ such that $HNP_{e_0}(E) = HNP_{\infty}(E)$. We show this by induction on rank. In the rank one case, there is nothing to show. 

Pick $0 < \epsilon \ll 1$. Then we have the Euclidean distance $\| p_{i_j, e} - p_{j, \infty} \| < \epsilon$ for all $j$ and all $e \gg 0$. Replacing $E$ by $(F^{e})^*E$ for some large $e$, we may assume that $\| p_{i_j, e} - p_{j, \infty} \| < \epsilon$ holds for all $e \geq 0$. 
Since the rank coordinates of $p_{i_j,e}$ and $p_{j,\infty}$ can only take integer values, it must be the case that $r_{j,\infty} = \mathrm{rk}(E_{i_j, e})$.

We first show that there exists an integer $e_0$ such that $E_{i_1, e} = (F^{e-e_0})^*E_{i_1, e_0}$ for all $e \geq e_0$. 
Consider the first line segment $s$ of the polygon $\HNP_\infty(E_{i_1,0})$. Note that $s$ cannot lie above the line segment $\overline{p_{0,\infty} p_{1,\infty}}$ (see Remark \ref{RemarkHNPLiesAboveSubsheaves}). We will show that $s$ actually lies on the line segment $\overline{p_{0,\infty} p_{1,\infty}}$. Assuming that is the case, we apply the induction hypothesis to $E_{i_1,0}$ to get:  there exists a subsheaf $G$ of $(F^l)^*{E_{i_1,0}}$ such that $(\mathrm{G},\frac{\deg{G}}{p^l}) \in s \subset \overline{p_{0,\infty} p_{1,\infty}}$ i.e.~$G$ is strongly semistable. Once again, we may replace $E$ by $(F^l)^*E$ to assume that $l = 0$. 
By induction hypothesis, since the theorem holds for $E/G$ and since $\mu(G) = \mu_{\max}(E) = \frac{d_{1,\infty}}{r_{1,\infty}}$, we also get the theorem for $E$. 

Now suppose that $s$ lies strictly below the line segment $\overline{p_{0,\infty} p_{1,\infty}}$. We will deduce a contradiction from this. Since $p_{i_1, l} \to p_{i_1, \infty}$ as $l \to \infty$, we can find an integer $l$ such that $\overline{p_{0,l}p_{i_1, l}}$ lies above $s$. Thus, $\mu_{\max}((F^{l})^*E) > \mu_{\max}((F^l)^*E_{i_1, 0})$ for such an integer $l$. 

Since $\{ (F^l)^*E_{i_1, 0},(F^l)^*E_{i_1 + 1, 0},\dots,(F^l)^*E \}$ form a filtration of $(F^l)^*E$, using Lemma \ref{SESMuMaxBound}, we get that there exists an integer $j > i_1$ such that  
\[\mu_{\max}((F^l)^*(E_{j,0}/E_{j-1,0})) > \mu_{\max}((F^l)^*E_{i_1, 0}).\] Consider a saturated subsheaf $G \subset (F^l)^*(E_{j,0})$ such that $\mu(G/(F^l)^*(E_{j-1,0})) = \mu_{\max}((F^l)^*(E_{j,0}/E_{j-1,0}))$. 
Then, $(\mathrm{rk}(G), \frac{\deg(G)}{p^l}) \in \HNP_{l}(E)$. Let us try to estimate the difference in the areas of $\HNP_l(E)$ and $\HNP(E)$. There is a lower bound on the difference in areas by considering the triangle $W$ joining the points $p_{j-1,0}$, $p_{j,0}$ and  $(\mathrm{rk}(G), \frac{\deg(G)}{p^l})$. Let $\tilde{\mu}_1 := \frac{\mu((F^l)^*(E_{j,0}/E_{j-1,0}))}{p^l}$ and $\tilde{\mu}_2 := \frac{\mu(G/(F^l)^*E_{j-1,0})}{p^l}$ denote the slopes of the sides of $W$ that contain $p_{j-1,0}$. Using the fact that the $\mathrm{rk}$-coordinates of these points can only take integer values, we can get a get a lower bound on the area of the triangle $W$ by considering the triangle $W'$ joining $p_{j-1,0}$, $p_{j-1,0} + (1,\tilde{\mu}_1)$, $p_{j-1,0} + (1,\tilde{\mu}_2)$ (see Figure \ref{fig:triangles}). 
\begin{figure}[h!]
\centering
\begin{tikzpicture}
\node at (0,0) [below]  {$p_{j-1,0}$}; 
\fill (0,0) circle (2pt); 

\node at (6,3) [right]  {$p_{j,0}$}; 
\fill (6,3) circle (2pt); 

\node at (3,5) [above]  {$(\mathrm{rk}(G),\frac{\deg(G)}{p^l})$}; 
\fill (3,5) circle (2pt); 

\node at (2,1) [below=2pt,right=2pt]  {$p_{j-1}+(1,\tilde{\mu}_1)$}; 
\fill (2,1) circle (2pt); 

\node at (2,3.33) [above,left]  {$p_{j-1}+(1,\tilde{\mu}_2)$}; 
\fill (2,3.33) circle (2pt); 

\node at (1.4,1.4) {$W'$};

\draw (0,0) -- (6,3);
\draw (0,0) -- (3,5); 
\draw (3,5) -- (6,3);
\draw (2,1) -- (2,3.33);
\end{tikzpicture}
    \caption{Triangles $W$ and $W'$ appearing in the proof of Theorem \ref{fdHN}. The larger triangle is $W$ and the smaller triangle inside it is $W'$.} 
    \label{fig:triangles} 
\end{figure}
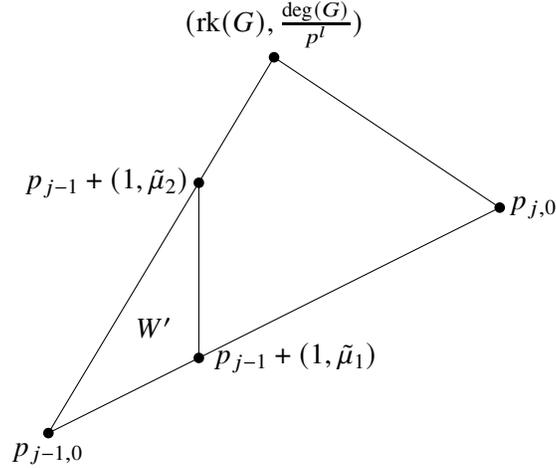

Thus it follows:
\begin{align*}
&\mathrm{Area}(\HNP_l(E)) - \mathrm{Area}(\HNP(E)) \\ &\geq  \mathrm{Area}(W') \\ &\geq \frac{1}{2}(\tilde{\mu}_2 - \tilde{\mu}_1) \\
    &= \frac{1}{2} \left( \frac{\mu(G/(F_X^l)^*(E_{j-1,0}))}{p^{l}} - \frac{\mu((F_X^l)^*(E_{j,0}/E_{j-1,0}))}{p^l}\right) \\
    &\geq \frac{1}{2} \left( \mu({E_{i_1, 0}}) - \mu(E_{j, 0}/E_{j-1,0}) \right) \\
    &\geq \frac{1}{2}((\mu_{1\infty} - \epsilon) - \mu(E_{j, 0}/E_{j-1,0}) ).
\end{align*}

To finish the estimate, we need to get an upper bound on $\mu(E_{j, 0}/E_{j-1,0})$. To do this, let $u$ be such that $i_u < j \leq i_{u+1}$. Then the line segment $\overline{p_{j-1,0}p_{j,0}}$ is sandwiched between the line segments $\overline{p_{i_u, 0}p_{i_{u+1}, 0}}$ and $\overline{p_{i_u,\infty} p_{i_{u+1},\infty}}$ (see Figure \ref{fig:lineSegment}). Using the fact that $\| p_{i_u,\infty} - p_{i_u 0} \| \leq \epsilon$ and estimating the slope of the line segment $\overline{p_{j-1,0}p_{j,0}}$, we get that $| \mu(E_{j, 0}/E_{j-1,0}) - \mu_{u+1,\infty} | \leq 3 \epsilon$. Thus, we have
\begin{figure}[h!]
    \centering
\begin{tikzpicture}
\node at (0,2) [above] {$p_{i_u,\infty}$};
\fill (0,2) circle (2pt); 

\node at (8,5) [above] {$p_{i_{u+1},\infty}$};
\fill (8,5) circle (2pt); 

\draw [red] (0,2) -- (8,5);

\node at (0,0) [below] {$p_{i_u,0}$};
\fill (0,0) circle (2pt); 

\node at (8,2) [below] {$p_{i_{u+1},0}$};
\fill (8,2) circle (2pt); 

\draw [dashed] (0,0) -- (8,2);

\node at (2,1.4) [above] {$p_{j-1,0}$};
\fill (2,1.4) circle (2pt); 

\node at (5,2.3) [above] {$p_{j,0}$};
\fill (5,2.3) circle (2pt); 

\draw [blue] (2,1.4) -- (5,2.3);
\draw [blue] (0,0) -- (0.6,0.6);
\draw [blue] (1.4,1.1) -- (2,1.4);
\draw [blue,dotted] (0.6,0.6) -- (1.4,1.1);

\draw [blue] (5,2.3) -- (6,2.3);
\draw [blue] (7,2.2) -- (8,2);
\draw [blue,dotted] (6,2.3) -- (7,2.2);

\draw [decorate,decoration={brace,amplitude=5pt}, xshift =-4pt ,yshift=0pt]
(0,0) -- (0,2)node [black,midway,xshift=-14pt] {$\leq \epsilon$};

\draw [decorate,decoration={brace,amplitude=5pt}, xshift =+4pt ,yshift=0pt]
(8,5) -- (8,2)node [black,midway,xshift=+14pt] {$\leq \epsilon$};
\end{tikzpicture}
    \caption{Estimating  $\mu(E_{j,0}/E_{j-1,0})$. The red line shows the boundary of $\HNP_\infty(E)$ and the blue line shows the boundary of $\HNP(E)$. }
    \label{fig:lineSegment}
\end{figure}
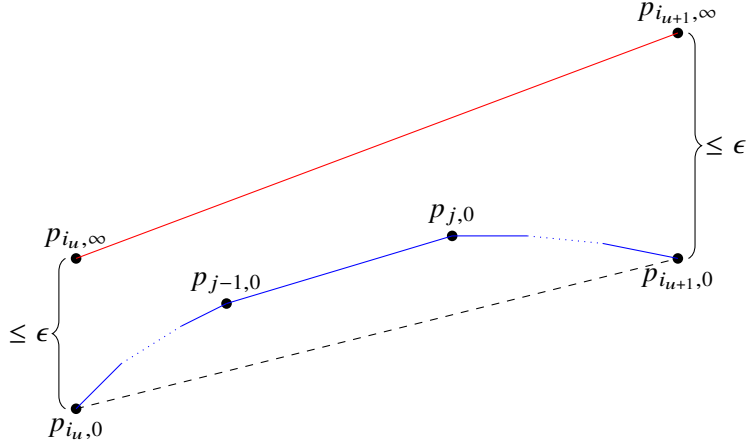

$$\mathrm{Area}(\HNP_l(E)) - \mathrm{Area}(\HNP(E)) \geq \frac{1}{2}(\mu_{1\infty} - \mu_{2\infty}
-4\epsilon).$$
But the difference in areas between $\HNP(E)$ and $\HNP_\infty(E)$ is at most $r\epsilon$, which gives us a contradiction for a small enough choice of $\epsilon$. 
\end{proof}

\section{Bogomolov's inequality}\label{Sec BI}
In this section, we give several equivalent statements on Bogomolov's inequality, about the positivity of $\Delta(E)\cdot D_2\cdots D_{n-1}$.

Fix a positive integer $r$.
Let $X$ be a smooth projective variety of dimension $n\geq 2$ over an algebraically closed field $k$ in any characteristic. Note that in the characteristic zero case, we define strong semistability to be the same as the semistability.
	We fix a nef divisor $A$ over $X$ such that $T_X(A)$ is globally generated.
	Define the constant $\beta_r=\beta_r(A;D_1,\ldots, D_{n-1})$ for a choice of divisors $(D_1,\ldots, D_{n-1})$ as below
	\[
	\beta_r(A;D_1,\ldots, D_{n-1}):= 
	\begin{cases}
		0, & \text{if}~\mathrm{char}(k)=0;\\
		\left(\frac{r(r-1)}{p-1} AD_1\cdots D_{n-1}\right)^2, & \text{if}~\mathrm{char}(k)=p.
	\end{cases}
\]
We also reference the \emph{discriminant} of a rank $r$ torsion-free sheaf $E$, denoted as $\Delta(E)$, is defined as 
\[
\Delta(E):= 2r\cdot  c_2(E)-(r-1)c_1(E)^2.
\]
When $G$ is a non-trivial subsheaf of rank $s$ in $E$, we set $$\xi_{G,E} = \frac{c_1(G)}{s} - \frac{c_1(E)}{r}.$$
We also use $K^+$ to denote an open cone in $\mathrm{Num}(X)$, defined as in the beginning of Subsection \ref{sub ii to iii}.

The main theorem is the following.
\begin{theorem}[BI($r$)]\label{BI}
Let $(D_1,\ldots, D_{n-1})$ be a collection of nef line bundles over $X$, and let $d=D_1^2\cdot D_2\cdots D_{n-1}\geq 0$. Fix a positive integer $r$.
Then the following statements are equivalent. 
\begin{enumerate}
    \item[(i)] Assume each $D_i$ is ample, and $E$ is a  strongly $(D_1,\ldots, D_{n-1})$-semistable torsion-free sheaf of rank $r'\leq r$.
		Then we have
		\[
		\Delta(E)D_2\ldots D_{n-1}\geq 0.
		\]
	\item[(ii)] Let $E$ be a $(D_1,\ldots, D_{n-1})$-semistable torsion-free sheaf of rank $r'$ for some $r'\leq r$.
		Then we have 
		\[
		d \cdot \Delta(E)D_2\cdots D_{n-1} + \beta_{r'} \geq 0.
		\]
	\item[(iii)] Let $E$ be a torsion-free sheaf of rank $r'$ for some $r'\leq r$ over $X$, and assume $d \cdot \Delta(E)D_2\cdots D_{n-1} + \beta_{r'} <0$.
	Then there exists a saturated subsheaf $E'$ of $E$ such that $\xi_{E',E}\in K^+$.
    \item[(iv)] Let $E$ be a strongly $(D_1,\ldots, D_{n-1})$-semistable torsion-free sheaf of rank $r'\leq r$ over $X$.
    Then 
    \[
    \Delta(E)D_2\cdots D_{n-1} \geq 0.
    \]
\end{enumerate}
\end{theorem}

\begin{remark}
In terms of the notations as in \cite{langer}, the above corresponds to $T^5(r) \Rightarrow T^3(r) \Rightarrow T^4(r) \Rightarrow T^2(r) \Rightarrow T^5(r)$.
Note that the last implication from $T^2(r)$ to $T^5(r)$ is trivial, as the statement (i) is the same as (iv) but with a stronger assumption.

\end{remark}
\begin{remark}
In characteristic zero, since $\beta_r=0$ and the strong semistability is the same as the semistability, the statement above can be more or less combined into one single statement about the non-negativity of $\Delta(E)D_2\cdots D_{n-1}\geq 0$.
This was first proved by Bogomolov in \cite{bogomolov}, thus the name.
\end{remark}

\subsection{$(i) \Rightarrow (ii)$}
In this subsection, our goal is to prove the implication $(i) \Rightarrow (ii)$ in Theorem \nameref{BI} for torsion-free sheaf $E$ of rank $r$, and we follow \cite[$\S$3.6]{langer}.

To prove the statement, we start with the following preparation.
Let $H$ be an ample line bundle over $X$, and let $(D_1,\ldots ,D_{n-1})$ be as in Theorem \nameref{BI} (ii).
Then $(H_1(t),\ldots, H_{n-1}(t)):=(D_1+tH, \ldots, D_{n-1}+tH)$ is a polarization of ample line bundles.
\begin{lemma}\label{T5(r) to T3(r) lem}
	Assume Theorem \nameref{BI} (i).
	Then for $t>0$ we have
	\begin{multline*}
	H_1(t)^2H_2(t)\cdots H_{n-1}(t)\cdot \Delta(E)H_2(t)\ldots H_{n-1}(t) + \\ r^2 (L_{max, t}(E)-\mu_{t} (E))(\mu_{t}(E)-L_{min, t}(E))\geq 0.
	\end{multline*}
\end{lemma}
Here we follow the notations $L_{max}$ and $L_{min}$ as in Subsection \ref{positive char}, and denote $L_{max, t}$, $L_{min, t}$ and $\mu_t$ as the ones defined using the polarization $(H_1(t),\ldots, H_{n-1}(t))$.
\begin{remark}\label{T5(r) to T3(r) rmk}
Before we prove the lemma, it is worth noting that the above inequality together with its proof, which only uses Theorem \nameref{BI} (i), applies also to a general polarization $(D_1,\ldots, D_{n-1})$ for ample $D_i$, in place of $(H_1(t),\ldots, H_{n-1}(t))$.
More explicitly, Theorem \nameref{BI}, (i) will imply the inequality
\[
D_1^2D_2\cdots D_{n-1}\cdot \Delta(E)D_2\cdots D_{n-1} + r^2(L_{\max}-\mu)(L_{\min}-\mu) \geq 0,
\]
where $D_i$ are all ample.
\end{remark}
\begin{proof}
		We first notice that by Theorem \ref{fdHN}, there exists $k\in \mathbb{N}$ such that all of the factors in the Harder-Narasimhan filtration of $(F^k)^*E$ are strongly semistable.
	Denote $0= E_0\subset \cdots E_m=(F^k)^*E$ to be the corresponding filtration of $(F^k)^*E$, and let $F_i$ be the quotient $E_i/E_{i-1}$, $r_i=\mathrm{rk}(F_i)$, and $\mu_{i,t}=\mu_t(F_i)$.
	Then by the Hodge index theorem,
	\footnote{Here we are using the form of the Hodge index theorem that $(D^2D_2\ldots D_{n-1}) \cdot (D_1^2D_2\ldots D_{n-1}) \leq (DD_1D_2\ldots D_{n-1})^2$, where $D_i$ are nef divisors satisfying $D_1^2D_2\cdots D_{n-1}>0$. This can be proved using the standard Hodge index theorem for ample divisors together with the approximation technique as in Subsection \ref{subsec approx}.} 
	we have
	\begin{align*}
		&\frac{\Delta((F^k)^*E) H_2(t)\cdots H_{n-1}(t)}{r} \\ & = \sum_i \frac{\Delta(F_i)H_2(t)\cdots H_{n-1}(t)}{r_i} - \frac{1}{r} \sum_{i<j} r_ir_j(\frac{c_1 (F_i)}{r_i} - \frac{c_1 (F_j)}{r_j})^2 H_2(t)\cdots H_{n-1}(t) \\
		& \geq \sum_i \frac{\Delta(F_i)H_2(t)\cdots H_{n-1}(t)}{r_i} - \frac{1}{rd}\sum_{i<j} r_ir_j(\mu_{i,t}-\mu_{j,t})^2.
	\end{align*}
The assumption of Theorem \nameref{BI} (i) and the ampleness of $H_i(t)$ provides us with the inequality
\[
\Delta(F_i)H_2(t)\cdots H_{n-1}(t) \geq 0.
\]
On the other hand, we have the following elementary inequality on $r_i$, $\mu_i$ (with $r=\sum_i r_i$ and $r\mu = \sum_i r_i\mu_i$) that
\begin{equation}
\label{EquationElementaryInequality}
\sum_{i<j} r_ir_j(\mu_i-\mu_j)^2 \leq r^2 (\mu_1-\mu)(\mu-\mu_m).
\end{equation}
Combine with the inequality above, we get the inequality in the statement of lemma.
\end{proof}

With the above lemma together with the approximation technique as in Subsection \ref{subsec approx}, we are ready to prove the aforementioned implication.
\begin{proof}[Proof of Theorem \nameref{BI}, $(i)\Rightarrow (ii)$]
	Let $H$ be a fixed ample line bundle.
	By Theorem \ref{approx}, we can find a filtration of torsion-free coherent subsheaves $0=E_0\subset E_1\subset \cdots \subset E_m=E$ of $E$ such that it is the Harder-Narasimhan filtration with respect to the polarizations $(H_1(t),\ldots, H_{n-1}(t))$, where  $t>0$ is very close to $0$.
	As $E$ is $(D_1,\ldots, D_{n-1})$-semistable, we have
	\begin{align*}
		\mu(E) & \geq \mu(E_1)\\
		& = \lim_{t \rightarrow 0^+} \mu_t(E_1)\\
		& \geq \lim_{t \rightarrow 0^+} \mu_t(E)\\
		& = \mu(E).
	\end{align*}
By the choice of $E_1$, the above implies the equality
\[
\mu(E)=\lim_{t \rightarrow 0^+} \mu_{{max},t}(E).
\]
Similarly, we have
\[
\mu(E) = \lim_{t \rightarrow 0^+} \mu_{{min},t}(E).
\]
On the other hand, Proposition \ref{bound of alpha E} states that for a nef divisor $A$ such that $T_X(A)$ is globally generated, we have
\begin{align*}
	L_{{max},t}(E)-\mu_{{max},t}(E),~\mu_{{min}, t}(E) - L_{{min}, t}(E) \leq \frac{r-1}{p-1} AH_1(t)\cdots H_{n-1}(t).
\end{align*}
In particular, the limit of the right hand side of the inequality above is equal to $\frac{r-1}{p-1}AD_1\cdots D_{n-1}$, whose square is $\frac{1}{r^2}\beta_r$.
In this way, apply the inequality of Lemma \ref{T5(r) to T3(r) lem} for each $t>0$, we get
\begin{align*}
		 &D_1^2D_2\cdots D_{n-1}\cdot \Delta(E)D_2\ldots D_{n-1} + \beta_r \\
		 &= \lim_{t \rightarrow 0^+} \Bigg( D_1^2D_2\cdots D_{n-1}\cdot \Delta(E)D_2\ldots D_{n-1}  + {r^2}\left(\frac{r-1}{p-1} AH_1(t)\cdots H_{n-1}(t)\right)^2 \Bigg)\\
		 &\geq  \lim_{t \rightarrow 0^+} \left( H_1(t)^2H_2(t)\cdots H_{n-1}(t) \cdot \Delta(E) H_2(t)\cdots H_{n-1}(t) \right) \\ 
		&~~~~~~~~~~~~~~~~~~~~~~~~~~~~+ r^2(L_{{max}, t}(E) -\mu_t(E)) (\mu_t(E) - L_{{min}, t}(E))  \\
        &\geq 0.
\end{align*}
\end{proof}


\subsection{$(ii)\Rightarrow (iii)$}\label{sub ii to iii}
In this subsection, we show the implication $(ii) \Rightarrow (iii)$, following \cite[$\S$3.7]{langer}

As a preparation,  we define an open cone in $\mathrm{Num}(X)$ as follows.
\begin{definition}\label{DefOfCone}
Let $K^+$ be the following open cone in $\mathrm{Num}(X)$:
\begin{multline*}
    K^+ = \{D\in \text{Num}(X): D^2D_2 \cdots D_{n-1} > 0, \text{and}  \ DD' D_2\cdots D_{n-1} \geq 0 \\ \text{for all nef } D'\}.
\end{multline*}
\end{definition}
The following description is used in the proof.

\begin{lemma}\label{DescripOfCone}
	A divisor $D \in K^+$ if and only if it satisfies the inequalities $DLD_2 \cdots D_n > 0$ for all $L \in \bar{K}^+\setminus \{0\}$.
\end{lemma}
Before we start, we also mention several elementary computational results whose proofs are left as exercises to the reader.
\begin{lemma}\label{DiscriminantSES}
Let $0 \to E' \to E \to E'' \to 0$ be a short exact sequence of coherent sheaves. Let $r'$ (resp. $r''$) be the rank of $E'$ (resp. $E''$), and assume they are positive. We have 
\begin{multline*}
    \frac{\Delta(E)D_2 \cdots D_{n-1}}{r} + \frac{rr'}{r''}\xi_{E', E}^2D_2 \cdots D_{n-1} = \frac{\Delta(E')D_2 \cdots D_{n-1}}{r'} + \\ \frac{\Delta(E'')D_2 \cdots D_{n-1}}{r''}
\end{multline*}
\end{lemma}
\begin{lemma}\label{xiSES}
Let $0 \to E' \to E \to E'' \to 0$ be a short exact sequence of coherent sheaves. Let $r'$ ($r''$) be the rank of $E'$ ($E''$) respectively.

\begin{itemize}
    \item If $G$ is a non-trivial subsheaf of $E'$, then 
    $$\xi_{G, E} = \xi_{E', E} + \xi_{G, E'}.$$
    \item If $G'' \subset E''$ is a proper subsheaf of rank $s$, and denote by $G$ the kernel of the map $E \to E''/G''$, then
    $$\xi_{G, E} = \frac{r'(r''-s)}{(r'+s)r''}\xi_{E', E} + \frac{s}{r'+s}\xi_{G'', E''}.$$
\end{itemize}
\end{lemma}

\begin{proof}[Proof of Theorem \nameref{BI}, $(ii)\Rightarrow (iii)$]
We prove (iii) inductively on $\mathrm{rk}(E)$ $=r$, assuming the statement (ii). 
As a starting point, we notice that when $E$ is of rank one, as $\Delta(E)=0$ and $\beta_1\geq 0$, the statement (iii) is automatically true.
In general, the assumption of Theorem \nameref{BI} (ii) tells us that $E$ is not semistable with respect to $(D_1, \ldots, D_{n-1})$. Thus we have the maximal destabilizing subsheaf $E'\subset E$ with respect to the polarization. Let $E'' = E/E'$, and $r', r''$ be the ranks of $E', E''$ respectively. Then by Lemma \ref{DiscriminantSES} we have 
\begin{multline*}
    \frac{\Delta(E)D_2 \cdots D_{n-1}}{r} + \frac{rr'}{r''}\xi_{E', E}^2D_2 \cdots D_{n-1} = \\ \frac{\Delta(E')D_2 \cdots D_{n-1}}{r'} + \frac{\Delta(E'')D_2 \cdots D_{n-1}}{r''}.
\end{multline*}
We also have 
$$ \frac{\beta_r}{r} \ge \frac{\beta_{r'}}{r'}+\frac{\beta_{r''}}{r''}.$$
So by multiplying $d=D_1^2D_2\cdots D_{n-1}$, either $\xi_{E', E}^2D_2 \cdots D_{n-1} >0$, 
or one of $d\Delta(E')D_2\cdots D_{n-1}+\beta_{r'}$ or $d\Delta(E'')D_2\cdots D_{n-1}+\beta_{r''}$ is negative. 

In the first case above, we claim that $\xi_{E',E}\in K^+$.
To see this, by the definition of $K^+$ and the assumption on $\xi_{E',E}$ it suffices to show that for any $L\in {K}^+\backslash \{0\}$ we have $\xi_{E',E}LD_2\cdots D_{n-1}>0$.
Note first the assumption of $E'$ and $E$ implies that the inequality is true for $L=D_1$, and it reduces to show the sign of the function $L \mapsto \xi_{E',E}LD_2\cdots D_{n-1}$ is not changing on $K^+$.
This then follows from the continuity of the function, the connectivity of $K^+$, and the  Hodge Index Theorem that 
\[
(\xi_{E',E}LD_2\cdots D_{n-1})^2 \geq  \xi_{E',E}^2D_2\cdots D_{n-1} \cdot L^2 D_2\cdots D_{n-1} >0,
\]
for $L\in K^+\backslash \{0\}$.

Suppose now $\xi_{E', E}^2D_2 \cdots D_{n-1} \leq 0$ and we are in the second case above. Note that since both $E'$ and $E''$ are of smaller ranks, we can apply the induction hypotheses of ranks to get the following dichotomy: 
\begin{itemize}
    \item There is a saturated subsheaf $G \subset E'$ such that $\xi_{G, E'} \in K^+$, or
    \item There is a saturated subsheaf $G'' \subset E''$ such that $\xi_{G'', E''} \in K^+$.
\end{itemize}
Applying Lemma \ref{xiSES}, we get in either case that $\xi_{G, E}$ is a positive linear combination of $\xi_{E', E}$ and some element $L$ in $K^+$.

To proceed, define the open subcone $\mathcal{C}(\xi) := \{D \in \overline{K^+}\setminus \{0\} : \xi \cdot DD_2\cdots D_{n-1} >0\} \subset \overline{K^+}$ for a given $\xi \in \text{Num}(X)$. We observe that by replacing $E'$ with $G$ in either case, we get strictly larger cones than before: $\mathcal{C}(\xi_{E', E}) \subsetneq \mathcal{C}(\xi_{G, E})$. 
Here the inclusion part is clear, as $\xi_{G,E}$ is the sum of $\xi_{E',E}$ with some $L\in K^+$, where the second term only contributes positive values in the product $\xi_{G,E}DD_2\cdots D_{n-1}$.
To see the inclusion is strict, as $\xi_{E',E}$ is not in $K^+$, by Lemma \ref{DescripOfCone} there exists some $L'\in \overline{K^+}\backslash\{0\}$ such that $\xi_{E',E}L'D_2\cdots D_{n-1}\leq 0$ but $LL'D_2\cdots D_{n-1}>0$.
On the other hand, as $E'$ is the maximal destabilizing subsheaf of $E$ with respect to the polarization $(D_1,\ldots,D_{n-1})$, we have $\xi_{E',E}D_1\cdots D_{n-1}>0$ and $LD_1\cdots D_{n-1}> 0$, where the latter follows as $L\in K^+$.
In this way, there exists some real number $t\in [0,1)$ such that $\xi_{E',E}(tD_1+(1-t)L')D_2\cdots D_{n-1}=0$, and thus
\begin{align*}
    \xi_{G,E} & (tD_1+(1-t)L') D_2\cdots D_{n-1} \\
    & = \xi_{E',E}(tD_1+(1-t)L')D_2\cdots D_{n-1} \\
    &+ L(tD_1+(1-t)L')D_2\cdots D_{n-1}\\
    & = 0 + L(tD_1+(1-t)L')D_2\cdots D_{n-1} \\
    & > 0 = \xi_{E',E}(tD_1+(1-t)L')D_2\cdots D_{n-1}.
\end{align*}
So the element $tD_1+(1-t)L'$ is in $\mathcal{C}(\xi_{G,E})$ but not in $\mathcal{C}(\xi_{E',E})$.
In this way, we get a sequence of strictly increasing subcones of $\overline{K^+}$ by replacing $E'$ by $G$, until we reach to the situation where $\xi_{E', E}^2 D_2 \cdots D_{n-1}>0$. 

It remains to show that this process terminates within finite replacements. First note that by choosing ample $\mathbb{R}$-basis $H_1, \cdots, H_\rho$, which are contained in $\mathcal{C}(\xi_{E', E})$ for $\text{Num}(X)$, 
\footnote{Note that since $\mathcal{C}(\xi_{E'E})$ is an open subcone in $K^+\subset \mathrm{Num}(X)$ that is non-empty (it contains $D_1$), such a basis of $H_i$ exists.}
we have $$\xi_{G, E} \in \frac{1}{r!} (\mathbb{Z}H_1 + \cdots +\mathbb{Z} H_\rho).$$ Furthermore, we also have 
$$0 < \xi_{G, E} H_j D_2 \cdots D_{n-1} < \mu_{\max}^j(E) - \mu^j(E),$$
for all $j = 1, 2, \cdots, \rho$, where $\mu^j$ denote the slopes with respect to $(H_j, D_2, \cdots, D_{n-1})$. Thus $\xi_{G, E}$ is in fact contained in a bounded  discrete hence finite subset of $\text{Num}(X)$.
\end{proof}

\subsection{$(iii)\Rightarrow(iv)$}
In this section, we follow \cite[$\S$3.8]{langer} to show the implication $(iii)\Rightarrow (iv)$.

We follow the definition of $\beta_r$ as in the beginning of this section, and the definition of the cone $K^+$ with respect to the polarization $(D_1,\ldots, D_{n-1})$ as in the last subsection.

By the definition of $\xi_{E',E}$ and $K^+$, the statement (iii) implies that $E$ is not semistable with respect to $(D_1,\ldots, D_{n-1})$.
\begin{remark}
The implication $(iii) \implies (iv)$ is easy in $\mathrm{char}(k) = 0$, by the definition of $K^+$ applying at $\xi_{E',E}D_1\cdots D_{n-1}$.
\end{remark}

\begin{proof}[Proof of Theorem \nameref{BI}, $(iii)\Rightarrow (iv)$]
Assume $E$ is a torsion-free sheaf such that $\Delta(E)D_2\cdots D_{n-1} <0$.
We will deduce that $E$ is not strongly semistable. 
We would like to apply Theorem \nameref{BI}, (iii) to the Frobenius pullback $(F^l)^*E$ of $E$.  Since $\Delta((F^l)^*E)=p^{2l}\Delta(E)$, we get that
\begin{multline*}
D_1^2 D_2 \cdots D_{n-1} \cdot \Delta((F^l)^*E)D_2\cdots D_{n-1} + \beta_r =  \\ 
D_1^2 D_2 \cdots D_{n-1}  \cdot p^{2l} \cdot \Delta(E)D_2\cdots D_{n-1} + \beta_r    
\end{multline*}
 is negative when $l$ is large enough. By the assumption of Theorem \nameref{BI}, (iii) we can find a saturated subsheaf $E'$ of $(F^l)^*E$ with $\xi_{E', (F^l)^*E}\in K^+$.
By the definition of $K^+$ as in Definition \ref{DefOfCone}, we have 
\[
\xi_{E', (F^l)^*E}D_1\cdots D_{n-1} >0,
\]
in particular, the pullback $(F^l)^*E$ is not semistable.
Hence the torsion-free sheaf $E$ itself is not strongly semistable.
\end{proof}


\section{Restriction to hypersurfaces and Bogomolov's inequality}
\label{sec Res}
In this section, we prove one of the main technical ingredients Theorem \ref{Res}, which controls the change of various numerical invariants as one passes to hypersurfaces. Following this, we complete the induction schema sketched in the introduction. We follow the proof as in \cite[$\S$3.5, $\S$3.9]{langer} and \cite{langer-erratum}.
Combining the section with the equivalent statements in Theorem \nameref{BI}, we finish the proof of all theorems in \cite[Section 3]{langer}.

Fix a positive integer $r$, and let $X$ be a smooth projective variety over $k$.
Consider the following statement.
\begin{theorem}[Res($r$)]\label{Res}
Let $E$ be a torsion-free sheaf of rank $r$ over $X$, and let $(D_1,\ldots, D_{n-1})$ be a collection of nef divisors over $X$ with $d:=D_1^2D_2\cdots D_{n-1}\geq 0$.

Assume that $D_1$ is very ample and the restriction of $E$ to a very general divisor $D \in |D_1|$ is not semistable (with respect to $(D_2|_D,...,D_{n-1}|_D)$). Let $\mu_i$ ($r_i$) denote the slopes (ranks) of the Harder-Narasimhan filtration of $E|_D$. 
Then $$\sum_{i<j} r_i r_j (\mu_i - \mu_j)^2 \leq d \Delta(E) D_2...D_{n-1} + 2r^2(L_{max} - \mu)(\mu-L_{min}).$$
\end{theorem}

Our goal this section is to show the following two implications
\[
Res(r) + BI(r-1) \Longrightarrow BI(r) \Longrightarrow Res(r+1).
\]
Note that since both Theorem \nameref{Res} and Theorem \nameref{BI} are empty and thus automatically true for $r=1$, the above induction process proves that both Res(r) and BI(r) hold true for all $r\geq 1$.


\subsection{$Res(r)$ $+$ $BI(r-1) \Rightarrow$ $BI(r)$}
In this subsection, we consider the implication $Res(r) + BI(r-1) \Rightarrow BI(r)$, following \cite[$\S$3.5]{langer} and \cite{langer-erratum}. This is achieved via an inductive argument on the dimension of $X$. We first prove the case when $X$ is a surface. Further, we do not need to assume BI($r-1$) for the base case.
\begin{proposition}\label{implication for surface}
Let $X$ be a surface. 
Assume Theorem \nameref{Res} holds for all torsion-free sheaves of rank $\leq r$, then Theorem \nameref{BI} holds for all torsion-free sheaves of rank $\leq r$ and all ample divisor $D_1$.
\end{proposition}
\begin{proof}
By the equivalence of statements in Theorem \nameref{BI}, it suffices to show the assumption implies the statement in Theorem \nameref{BI}, (i), i.e. the inequality $\Delta(E)\geq 0$, for a strongly $D_1$-semistable torsion-free sheaf $E$ of rank $\leq r$, where $D_1$ is ample.

    Let $E$ be a rank $r$ vector bundle that is strongly $D_1$-semistable. Suppose by way of contradiction that $\Delta(E) < 0$. We may further assume that $E$ is locally free, by the inequality $$\Delta(E^{**}) = \Delta(E) - 2 r \operatorname{length}(E^{**}/E),$$ where the reflection $E^{**}$ is locally free on $X$. We first claim that under the assumption, the restriction to a general curve $C \in |D_1|$ is also strongly semistable. 
    To see this, since $E$ is strongly semistable, $\mu = L_{max} = L_{\min}$, and since (by assumption) $\Delta(E) < 0$, from the assumption we would get a contradiction $$0 \leq  \sum_{i<j} r_i r_j (\mu_i - \mu_j)^2 <0,$$ where we use the notation in \nameref{Res}. Similarly, as $(F^k)^*E$ is semistable, its restriction to a general curve $C$ is also semistable. Thus, the restriction of $E$ to a very general curve $C$ is strongly semistable. For the remainder of the proof, fix such a very general $C$.
    
    Now consider the symmetric power $S^{kr}E|_C$. 
    The strategy is to compute $\chi(S^{kr}E)$ in two different ways and use them to deduce a contradiction. 
    Since symmetric powers of strongly semistable sheaves on curves are strongly semistable (see \cite[Corollary 3.2.10]{huybrechts} for the characteristic zero case and \cite[Sections 3 and 5]{miyaoka} for the positive characteristic case), it follows that $S^{kr}E|_C$ is strongly semistable. Consider the short exact sequence arising from $C$: $$0 \to S^{kr} E(-k c_1(E)-C) \to S^{kr}E(-kc_1(E)) \to S^{kr}E(-k c_1(E))|_C \to 0.$$ This allows us to estimate $$h^0(S^{kr}E(-kc_1(E)) \leq h^0(S^{kr} E(-k c_1(E)-C)) + h^0(S^{kr}E(-k c_1(E))|_C).$$ Consider the first term on the right hand side. This quantity is equal to $\dim\left( \operatorname{Hom}(\mathcal{O}_X(kc_1(E) + C),S^{kr}E)\right)$, and we claim that this number is zero. 
    Note that the slope of $\mathcal{O}_X(kc_1(E) + C)$ is given by $k r \mu(E) + D_1^2$, and by the splitting principle,  the slope of the symmetric power can be seen to be $$ \mu(S^{kr}(E)) = \frac{\binom{r+kr-1}{kr} k c_1(E) D_1}{ \binom{r+kr-1}{kr}} = kr \mu(E).$$ 
    Since $D_1$ is ample, $D_1^2>0$, so the claim follows as there are no morphisms from semistable sheaves of higher slope to those of lower slope. 
    Thus we get $h^0(S^{kr}E(-kc_1(E)) \leq h^0(S^{kr}E(-k c_1(E))|_C)$. 
    
    Given any semistable vector bundle $G$ on a curve, we have the further estimate $h^0(G) \leq \max \{0,\operatorname{deg}G + \operatorname{rk}G\}$. 
    Using the sequence
    \[
    0 \longrightarrow G(-lP) \longrightarrow G \longrightarrow G|_{lP} \longrightarrow 0,
    \]
    the above inequality can be seen from the estimation $h^0(G) \leq  h^0(G \otimes \mathcal{O}_{C}(-lP)) + \mathrm{rk} (G) l = 0 + \mathrm{rk} (G) l \leq \max \{0, \mathrm{rk}G+\mathrm{deg}G\}$, where $P$ is a point on $C$, $l = \max \{0,\lceil\mu(G)\rceil \}$, and $h^0(G \otimes \mathcal{O}_{C}(-lP)) = 0$ as $\mu(\mathcal{O}_C(lP)) > \mu(G)$.
    
    Thus, we have $h^0(S^{kr}E(-kc_1(E)) = O(k^r)$. By Serre duality, we have the same order of magnitude estimate for $h^2(S^{kr}E(-kc_1(E))$. On the other hand, the splitting principle for Chern classes and Hirzebruch-Riemann-Roch theorem can be used to get the following estimate (for example, see \cite[Section 10]{bogomolov})  $$\chi(X,S^{kr}E(-kc_1(E)) = - \frac{r^r \Delta(E)}{2(r+1)!}k^{r+1} + O(k^r).$$ Since $\Delta(E)<0$, this polynomial is eventually positive and is of order $k^{r+1}$, but this is an obvious contradiction as $h^0$ and $h^2$ are of order $k^r$. 
\end{proof}

To show the implication for general $X$, we induct on the dimension of $X$, using the above surface case as the base case. 
However, different from the assumption in Proposition \ref{implication for surface}, for this to work we require the additional assumption of $BI(r-1)$ (i.e. Theorem $BI(r-1)$).
Precisely, we want to show that assuming 
\begin{itemize}
    \item $BI(r-1)$ and $Res(r)$ for $X$ of any dimension;
    \item $BI(r)$ for varieties of dimension $<n$,
\end{itemize}
then we have $BI(r)$ holds for any variety $X$ of dimension $n$.

%
\begin{proof}[Proof of Theorem $BI(r)$]
Similar to the proposition above, we aim to prove Theorem \nameref{BI}, (i) for $X$ of dimension $n$, namely the inequality $\Delta(E)D_2\cdots D_{n-1} \geq 0$ for a strongly $(D_1,\ldots,D_{n-1})$-semistable torsion-free sheaf $E$ over $X$, where $D_i$ are all ample.

Assume to the contrary that $\Delta(E)D_2\dots D_{n-1} < 0$. 
For a general divisor in $|D_2|$ (which we also denote as $D_2$), we have $\Delta(E|_{D_2})D_3\dots D_{n-1} = \Delta(E)D_2\dots D_{n-1} < 0$. 
Consider the polarization $(D_2^2, D_3,\ldots, D_{n-1}):=(D_2,D_2,D_3,\ldots,D_{n-1})$.
We first assume that $E$ is strongly $(D^2_2, D_3,  \dots D_{n-1})$-semistable. 
Applying $Res(r)$ to the Frobenius twists $(F^k)^*E$ with respect to $(D_2^2,\ldots, D_{n-1})$, then if $(F^k)^*E|_{D_2}$ is not semistable we would get the inequality
\[
\sum_{i<j} r_ir_j (\mu_i-\mu_j)^2 \leq dp^k\Delta(E)D_2\cdots D_{n-1} + 0 <0,
\]
which is impossible.
Thus $E|_{D_2}$ is strongly semistable for ample divisors $(D_2,\ldots D_{n-1})$, and by the induction hypothesis, since $E|_{D_2}$ is defined over a hypersurface of $X$ we have $\Delta(E|_{D_2})D_3\dots D_{n-1} = \Delta(E)D_2\dots D_{n-1} \geq 0$, a contradiction.

For the rest, we assume $E$ is not strongly $(D^2_2,D_3,\ldots, D_{n-1})$-semistable.
The trick is to interpolate between the polarizations $(D_1, D_2, \dots, D_{n-1})$ and $(D_2^2,D_3, \dots, D_{n-1})$. 
Let $B_t$ be the product $((1-t)D_1 + tD_2)D_2\dots D_{n-1}.$ 
By assumption, $E$ is strongly $B_0$-semistable while it is not strongly $B_1$-semistable. 
Since not being strongly semistable is an open condition with respect to $t$, there exists a number $t_k\in [0,1)$ such that $E$ is strongly $B_{t_k}$-semistable but not strongly $B_{t}$-semistable for  all $t_k < t \leq 1$. 

Let $k$ be a positive integer such that $(F^k)^*E$ is not semistable for $B_1$.
Using Theorem \ref{approx}, we see that the Harder-Narasimhan filtration of $E$ with respect to $B_t$ remains constant for all $t_k < t < t_k + \epsilon$ for some $\epsilon > 0$ small enough. 
We let $E'$ denote the maximal destabilizing sheaf of $(F^k)^*E$ with respect to the polarization $B_{t}$ for $t_k < t < t_k + \epsilon$. 
By the continuity of the slope with respect to $t$, it follows that $E'$ and $(F^k)^*E$ have the same $B_{t_k}$-slope (see Lemma \ref{semi approx}).
Since $E'$ is a subsheaf of a semistable sheaf with the same slope, $E'$ is also semistable with respect to $B_{t_k}$. 
It also follows that the quotient $E'' = (F^k)^*E/E'$ is $B_{t_k}$-semistable as well. 
Let $r$, $r'$ and $r''$ denote the ranks of $E$, $E'$ and $E''$ respectively. 

Now we apply Lemma \ref{DiscriminantSES} to the short exact sequence $E' \to E \to E''$ to get that
\begin{multline*}
  \frac{\Delta((F^k)^*E)D_2\dots D_{n-1}}{r} = \frac{\Delta(E')D_2\dots D_{n-1}}{r'} + \frac{\Delta(E'')D_2\dots D_{n-1}}{r''} - \\ \frac{r'r''}{r}\xi_{E',E''}^2D_2\dots D_{n-1}.  
\end{multline*}
Note that by the Hodge index theorem, we have 
$$\xi_{E',E''}^2D_2\dots D_{n-1} \cdot (t_kD_1 + (1-t_k)D_2)^2 D_2\dots D_{n-1} \leq (\xi_{E',E''}B_{t_k})^2.$$
Since $E',E''$ have the same $B_{t_k}$-slope, $(\xi_{E',E''}B_{t_k})^2 = 0$ and by ampleness of $D_1,\dots,D_{n-1}$, we have that 
$$d(t_k) := (t_kD_1 + (1-t_k)D_2)^2 D_2\dots D_{n-1} > 0.$$

Moreover by Theorem \nameref{BI}, (ii) but for $(r-1)$ in place of $r$, we get that the inequalities as below
\[
\frac{\Delta((F^k)^*E)D_2\dots D_{n-1}}{r} \geq -\frac{1}{d(t_k)}\left(\frac{\beta_{r'}(t_k)}{r'} + \frac{\beta_{r''}(t_k)}{r''}\right) 
\geq - \frac{\beta_r(t_k)}{rd(t_k)}, 
\]
where we denote $\beta_r(t)$ to be $\beta_r((1-t)D_1,+tD_2,D_2,\ldots, D_{n-1})$.
In this way, we get
\[
\Delta(E)D_2\cdots D_{n-1} \geq -\frac{\beta_r(t_k)}{d(t_k)p^{2k}}.
\]
Notice that the function $\frac{\beta_r(t)}{d(t)}$ for $t\in [0,1]$ is defined and continuous, thus bounded.
Taking the limit as $k$ approaches infinity, we get the inequality
\[
\Delta(E)D_2\cdots D_{n-1} \geq 0,
\]
which is a contradiction, so we are done.
\end{proof}
\begin{remark}
Pointed out by the referee, the last paragraph of the proof above can be argued as follows.
As $E'$ and $E''$ are of the same $B_{t_k}$-slopes as the strongly $B_{t_k}$-semistable sheaf $E$, both of them are strongly $B_{t_k}$-semistable.
Thus applying the induction hypothesis of Theorem \nameref{BI}, (i) to $E'$, $E''$ for the polarization $B_{t_k}$, we see
\[
\frac{\Delta(E')D_2\dots D_{n-1}}{r'} + \frac{\Delta(E'')D_2\dots D_{n-1}}{r''} \geq 0.
\]
In this way, by the same vanishing of $\xi_{E',E''}D_2\cdots D_{n-1}$ proved above, we get the positivity of $\Delta(E)D_2\cdots D_{n-1}$.
\end{remark}

\subsection{$BI(r) \Rightarrow Res(r+1)$}
We then consider the implication $BI(r) \Rightarrow Res(r+1)$, following \cite[$\S$3.9]{langer}.

Let $r$ be a non-negative integer, and $(D_1,\ldots, D_{n-1})$ be a nef polarization over $X$ with $d=D_1^2D_2\cdots D_{n-1}\geq 0$.

\begin{proof}
	Let $\Pi$ be the projective space associated to the linear system $|D_1|$.
	Let $Z\subset \Pi\times X$ be the incident scheme, defined as the locus $\{(D,x)\in \Pi \times X~|~x\in D\}$, where $p:Z \rightarrow \Pi$ and $q:Z \rightarrow X$ are the two natural maps induced by the projections.
	For each $s\in \Pi$, we let $Z_s$ be its fiber along the map $q$, which is a divisor in $X$.
	Let $0= E_0\subset E_1 \subset \cdots \subset E_m=q^*E$ be the \HNF of $q^*E$ associated to the polarizaion 
	\begin{multline*}
	(p^*\mathcal{O}_\Pi(1)^{\dim \Pi}, q^*D_2,\ldots, q^* D_{n-1})= \\ (p^*\mathcal{O}_\Pi(1),\ldots, p^*\mathcal{O}_\Pi(1), q^*D_2,\ldots, q^* D_{n-1})
	\end{multline*} 
	over $Z$.
	Denote $F_i$ to be the subquotient $E_i/E_{i-1}$.
	Here we observe that the pull back of the  filtration $\{E_i\}$ at a  general fiber $Z_s$ for $s\in \Pi$ is the  \HNF of $E|_{Z_s}$ at the hypersurface $Z_s\in |D|$.
	Indeed, the slope of $q^*E$ with respect to $(p^*\mathcal{O}_\Pi(1)^{\dim \Pi}, q^*D_2,\ldots, q^* D_{n-1})$ is the same as the slope of $q^*E|_{Z_s}$. This shows that the filtration $\{E_i\}$ is also the relative Harder-Narasimhan filtraion for $q^*E$ with respect to the map $p$.
	In particular, we have the equalities $r_i=\mathrm{rk} F_i$ and $\mu_i=\mu(F_i)$, where the latter is with respect to the polarization given by $(p^*\mathcal{O}_\Pi(1)^{\dim \Pi}, q^*D_2,\ldots, q^* D_{n-1})$.
	
	Now let $\mathbb{P}^1\cong \Lambda\subset \Pi$ be a sufficiently general pencil, corresponding to a linear subsystem of $|D|$ in $X$ parametrized by a line $\mathbb{P}^1\subset \Pi$.
	Let $B\subset X$ be the base of $\Lambda$, which is the intersection of $Z_s$ for any two (hence all) $s\in \Lambda$, and is of codimension $2$ in $X$.
	Denote by $Y$ the incident scheme for $\Lambda$, which is equal to the closed subscheme $p^{-1}\Lambda$ inside of $Z$.
	Here by the construction of the incident scheme for a pencil, the projection map $Y=p^{-1}\Lambda\rightarrow X$ coincides with the blowup of $X$ at the base locus $B$.
	Moreover, depending on the dimension of $X$, we can write the first Chern class of the restriction $F_i|_Y$ in terms of the following:
	\begin{itemize}
		\item If $\dim(X)=2$, then $B$ consists of a union of a finite amount of points, and we may write each $c_1(F_i|_Y)$ as $M_i + \sum_j b_{ij} N_j$, with $M_i\in \mathrm{Pic}(X)$ and $N_j$ being the $j$-th exceptional divisor for $1\leq j\leq l$.
		We define $b_i$ to be the number $\frac{\sum_j b_{ij}}{l}$.
		\item If $\dim(X)\geq 3$, then  by Bertini's theorem the blowup center $B$ is a smooth connected closed subscheme of codimension two in $X$, and we may write each $c_1(F_i|_Y)$ as $M_i+b_i N$, for $M_i \in \mathrm{Pic}(X)$ and $N$ being the exceptional divisor.
	\end{itemize}
In any of the above cases, with respect to the polarization $$(p^*\mathcal{O}_\Lambda(1),q^*D_2,\ldots, q^*D_{n-1})$$ over $Y$, we can write the slope $\mu_i$ as the following
\[
\mu_i=\frac{M_iD_1\cdots D_{n-1} + b_id}{r_i}.\tag{0}
\]
Here as the collection of torsion-free sheaves $\{F_i\}$ forms the graded pieces of a filtration of $q^*E$, we have 
\[
\sum_i b_i=c_1(q^*E)p^*\mathcal{O}_\Lambda(1)q^*D_2\cdots q^*D_{n-1}=0. \tag{1}
\]
Furthermore, since $(q|_Y)_*(E_i|_Y) \subset E$, we have the inequalities
\[
\frac{\sum_{j\leq i} M_jD_1\cdots D_{n-1}}{\sum_{j \leq i} r_j} \leq \mu_{\max}, \text{ for all } i,
\]
which is equivalent to the following inequalities
\[
\sum_{j\leq i} b_jd \geq \sum_{j\leq i} r_j(\mu_j-\mu_{\max}).\tag{2}
\]

Finally, we are ready to prove the implication.
As each $F_i$ is of rank $\leq r-1$, by Theorem \nameref{BI}, (i)  for the polarization $(p^*\mathcal{O}_\Pi(1)^{\dim \Pi}, q^*D_2,\ldots, q^* D_{n-1})$ over $Z$, we have
\begin{multline*}
\Delta(F_i)p^*\mathcal{O}_\Pi(1)^{\dim \Pi}q^*D_2\cdots q^* D_{n-1} = \\ \Delta(F_i|_Y)p^*\mathcal{O}_\Lambda(1)q^*D_2\cdots q^*D_{n-1} \geq 0,\text{ for all } i.
\end{multline*}
Applying this, we get
\begin{align*}
	& \frac{d\Delta(E)D_2\cdots D_{n-1}}{r} \\
	& =\sum_i \frac{d\Delta(F_i|_Y)p^*\mathcal{O}_\Lambda(1)q^*D_2\cdots q^*D_{n-1}}{r_i} \\ 
	& - \frac{d}{r}\sum_{i<j}r_ir_j\left(\frac{c_1(F_i|_Y)}{r_i} - \frac{c_1(F_j|_Y)}{r_j}\right)^2 (q|_Y)^*(D_2\cdots D_{n-1}) \\
	&\geq - \frac{d}{r}\sum_{i<j}r_ir_j\left(\frac{c_1(F_i|_Y)}{r_i} - \frac{c_1(F_j|_Y)}{r_j}\right)^2 (q|_Y)^*(D_2\cdots D_{n-1})\\
	&= \frac{d}{r} \sum_{i<j} r_ir_j\left( d(\frac{b_i}{r_i} -\frac{b_j}{r_j})^2 - (\frac{M_i}{r_i} - \frac{M_j}{r_j})^2D_2\cdots D_{n-1} \right) \\
	&\geq \frac{1}{r}\sum_{i<j} r_ir_j \left( d^2 (\frac{b_i}{r_i} - \frac{b_j}{r_j})^2 - (\frac{M_iD_1\ldots D_{n-1}}{r_i} - \frac{M_jD_1\cdots D_{n-1}}{r_j})^2 \right),
\end{align*}
where the first inequality is the application of the statement (i) as above, the equality after is a rearrangement using the formula (0), and the last inequality follows from the Hodge index theorem applied to $\frac{M_i}{r_i}-\frac{M_j}{r_j}$ and the polarization $D_1, \cdots, D_{n-1}$.
Moreover, using the formula (0), we may rewrite the last expression of the inequalities above as 
\[
\frac{2d}{r}\sum_{i<j}(\mu_i-\mu_j)(b_ir_j-b_jr_i) - \frac{1}{r}\sum_{i<j} r_ir_j(\mu_i-\mu_j)^2
\]
Using (1), one further simplifies the above to 
\[
2\sum_i db_i\mu_i - \frac{1}{r}\sum_{i<j} r_ir_j(\mu_i - \mu_j)^2.
\]

Here more concretely, the simplification above can be seen via an induction on $m$. If $m=1$, this is obvious. Suppose this is true for $m-1$. Then
\begin{align*}
    &\frac{2d}{r}\sum_{i<j}(\mu_i-\mu_j)(b_ir_j-b_jr_i)\\
    &= 2\sum_{i<j\leq m-1} (\mu_i-\mu_j)((b_i+\frac{b_m}{m})r_j-(b_j+\frac{b_m}{m})r_i)\\
    &+\frac{b_m}{m}\sum_{i<j\leq m-1}(\mu_i-\mu_j)(r_i-r_j)
    + \frac{2d}{r}\sum_{i\leq m-1}(\mu_i-\mu_m)(b_ir_m-b_mr_i)\\
    &=\frac{2(r-r_m)}{r}\sum_{i\leq m-1} d(b_i+\frac{b_m}{m})\mu_i \\
    &+ \frac{2d}{r}\left(r_m\sum_{i\leq m-1} b_i\mu_i-b_m(\mu r - \mu_mr_m)+b_m\mu_mr_m+b_m\mu_m(r-r_m) \right)\\
    &+\frac{b_m}{m}\sum_{i<j\leq m-1}(\mu_i-\mu_j)(r_i-r_j)\\
    &= 2\sum_i db_i\mu_i
\end{align*}
where the second equality follows from induction hypothesis.
Furthermore, using the formula (2) together with an elementary equality 
\[
\sum_i a_ib_i = \sum_i (\sum_{j\leq i}a_j)(b_i - b_{i+1}),
\] we get
\begin{align*}
	\sum_i db_i\mu_i & = \sum_i (\sum_{j\leq i} db_j) (\mu_i - \mu_j) \\
	& \geq \sum_i \left( \sum_{j \leq i} r_j (\mu_j -\mu_{\max})(\mu_i -\mu_{i+1}) \right) \\
	& = \sum_i ( r_i\mu_i^2 - r\mu\mu_{\max}) \\
	& \geq \sum_i r_i\mu_i^2 -r\mu^2 + r(\mu -\mu_{\max})(\mu - \mu_{\min}) \\
	& = \sum_{i<j} \frac{r_ir_j}{r}(\mu_i - \mu_j)^2 + r(\mu - \mu_{\max})(\mu - \mu_{\min}).
\end{align*}
As a consequence, we get
\[
	\frac{d\Delta(E)D_2\cdots D_{n-1}}{r}  \geq \sum_{i<j} \frac{r_ir_j}{r}(\mu_i - \mu_j)^2 + 2r(\mu-\mu_{\max})(\mu- \mu_{\min}).
	\]
At last, by moving the second term in the right hand side above and the inequalities $L_{\max}\geq \mu_{\max}$, $\mu_{\min}\geq L_{\min}$, we get the one as in Theorem \nameref{Res}.
\end{proof}


\section{Boundedness of torsion-free sheaves}\label{Sec bound}
At last, we use the ingredients in the last two sections to show the boundedness, following \cite[Section 4]{langer}

We start with the following result combining the inequalities in the last two sections into a uniform one, following \cite[Corollary 3.11]{langer}.
\begin{corollary}
\label{RestrictionLEstimate}
Let $E$ be a torsion-free sheaf on $X$ and $D_1$ be very ample, $D_2,\dots,D_{n-1}$ be ample, and let $D \in |D_1|$ be a general divisor. Then, we have that 
$$
\frac{r}{2}(L_{\max}(E|_D)- L_{\min}(E|_D))^2 \leq d \Delta(E)D_2\dots D_{n-1} + 2r^2 (L_{\max} - \mu)(\mu - L_{\min})
$$
\end{corollary}
\begin{proof}
First consider the case when $E|_{D}$ is not strongly semistable. Then, the inequality follows from Theorem \nameref{Res} as well as an elementary inequality \cite[Lemma 1.3]{langer} as below:
\[
\sum_{i<j} r_ir_j (\mu_i-\mu_j)^2 \geq \frac{r_1r_m}{r_1+r_m} r(\mu_1-\mu_m)^2 \geq \frac{r}{2}(\mu_1-\mu_m)^2
\]
where $r_1,\dots,r_m$ are positive real numbers and $\mu_1,\dots,\mu_m$ are real numbers and $r = r_1+ \dots + r_m$. 

Now consider the case when $E|_D$ is strongly semistable. Then, the left-hand side of the inequality is just zero and the result follows from the inequality in Remark \ref{T5(r) to T3(r) rmk} (where the latter is proved using Theorem \nameref{BI}, (i)).
\end{proof}

The boundedness is deduced from the following result. A more precise version is given in \cite{langer}.
\begin{theorem}
\label{RestrictionSlopeEstimate}
Let $H_1,\dots,H_{n-1}$ be very ample divisors on $X$ and let $X_l = |H_1| \cap \dots \cap |H_l|$, $1 \leq l \leq n-1$ be very general complete intersections.  Pick a nef divisor $A$ such that $T_{X_l}(A)$ is globally generated for all $0 \leq l \leq n-1$. Set $\beta_r = \beta(r;A,H_1,\dots,H_{n-1})$. Let $\mu_{\max,l}, \mu_{\min,l}$ denote the maximal and minimal slopes of the Harder-Narasimhan filtration of $E|_{X_l}$.
Then we have the following inequality,
\begin{multline*}
  \mu_{\max,l} - \mu_{\min,l} \leq \\ C(r,n) \left( \sqrt{\max\{d\Delta(E)H_2 \dots H_{n-1},0 \}} + \sqrt{\beta_r} + (\mu_{\max} - \mu_{\min})\right),  
\end{multline*}
where $C(r,n)$ is a constant depending only on $r$ and $n$.
\end{theorem}
\begin{remark}
Before we move on to the proof, it is worth mentioning that except for $\mu_{\max}-\mu_{\min}$, the rest of terms in the right hand side together with $\mu_{\min,l}$ can all be bounded using constants depending only on $X$ and the Hilbert polynomial of $E$.
\end{remark}
\begin{proof}
For $l = 1$, we use Corollary \ref{RestrictionLEstimate}, and get 
\begin{align*}
  (L_{\max,1} - L_{\min,1})^2 &\leq \frac{2}{r} d \Delta(E)H_2 \dots H_{n-1} + 4r(L_{\max} - \mu)(\mu - L_{\min})  \\
  &\leq \frac{2}{r} \max\{d\Delta(E)H_2 \dots H_{n-1},0\} + 4r(L_{\max} - L_{\min})^2 \\
  &\leq 4r \left( \max\{d\Delta(E)H_2 \dots H_{n-1},0\} + (L_{\max} - L_{\min})^2 \right)
\end{align*}

Now using the definition and Proposition \ref{bound of alpha E} that $\mu_{\max} - \mu_{\min} \leq L_{\max} - L_{\min} \leq \mu_{\max} - \mu_{\min} + \frac{2\sqrt{\beta_r}}{r}$, we get 
\begin{multline*}
  (\mu_{\max,1} - \mu_{\min,1})^2 \\
 \leq 4r \left( \max\{d\Delta(E)H_2 \dots H_{n-1},0\} + (\mu_{\max} - \mu_{\min} + 2\sqrt{\beta_r})^2 \right)
\end{multline*}
Using $\sqrt{a+b} \leq \sqrt{a} + \sqrt{b}$ for $a,b \geq 0$, we get 
\begin{multline*}
  \mu_{\max,1} - \mu_{\min,1} \\ \leq 4\sqrt{r} \left( \sqrt{\max\{d\Delta(E)H_2 \dots H_{n-1},0\}} + (\mu_{\max} -\mu_{\min}) + \sqrt{\beta_r} \right),
\end{multline*}
which gives the required inequality for $l = 1$. For higher $l$, we use induction. 
\end{proof}

Now we are ready to prove the main theorem of the article, which states that the moduli space, parametrizing torsion-free sheaves with a fixed Hilbert polynomial, and with a upper-bound for the slopes of their Harder-Narasimhan filtrations, is bounded.
\begin{theorem}
\label{BoundednessTheorem}
Let $X$ be a projective variety over an algebraically closed field with an ample line bundle $\mathcal{O}_X(1)$. Let $P$ be a polynomial of degree $d$ and let $\mu_0 \in \mathbb{R}$. Then, the family of torsion-free sheaves whose Hilbert polynomial are $P$ and whose $\mu_{\max}$ are at most $\mu_0$ is bounded. 
\end{theorem}
\begin{proof} 
We use the same notation as in Theorem \ref{RestrictionSlopeEstimate}.
According to  Kleiman's criteria (Theorem \ref{Kleiman}) and Lemma \ref{H0BoundedByMuMax}, to prove boundedness, it is enough to give an upper bound of $\mu_{\max,l}$ (or more specifically for $l=1$). Note that $\mu_l = \mu(E|_{X_{l}})$ is independent of $E$ and depends only on $P$ for all $0 \leq l \leq n$. To get such an upper bound, we use Theorem \ref{RestrictionSlopeEstimate} to get 
\begin{align*}
    \mu_{\max,l}(E) &\leq
\mu_l + (\mu_{\max,l}(E) - \mu_{\min,l}(E)) \\ &\leq C_1 + C_2(\mu_{\max} (E) - \mu_{\min}(E)) \\ &\leq C_1 + C_2 r (\mu_0 - \mu),
\end{align*}
where $C_1, C_2$ are constants independent of $E$.
Here the last inequality follows from $\mu_{\max}\leq \mu_0$ and the observation that $r\mu \leq \mu_{\min}(E) + (r-1) \mu_{\max}(E)$.
\end{proof}

A more general version of the above theorem for pure-dimensional sheaves can be proved by imitating the proof in \cite[Theorem 1.1]{simpson}.

\bibliographystyle{alpha} 
\bibliography{references}

\end{document}